\documentclass[a4paper, 11pt]{article}
\usepackage[utf8x]{inputenc}
\usepackage{graphicx, amsfonts,amsmath,amssymb,amsthm}
\usepackage{natbib}
\usepackage{hyperref,cleveref}
\usepackage{dsfont}

\usepackage{fullpage}
\usepackage{graphicx} 

\usepackage{enumitem}
\usepackage{subcaption}
\usepackage{commands}
\usepackage{tikzSettings}
\usepackage{longtable}
\usepackage{booktabs}
\usepackage{siunitx}
\author{Axel Parmentier$^1$, Victor Cohen$^1$, Vincent Leclère$^1$, Guillaume Obozinski$^2$, Joseph Salmon$^3$}
\title{Integer programming on the junction tree polytope for influence diagrams}
\date{%
    $^1$Universit\'e Paris-Est, CERMICS (ENPC), F-77455 Marne-la-Vall\'ee, France\\%
    $^2$ Swiss Data Science Center, EPFL \& ETH Zürich, Switzerland \\%
    $^3$ IMAG, Univ Montpellier, CNRS, Montpellier, France \\[2ex]%
    \today}

\begin{document}
\maketitle

\begin{abstract}


Influence Diagrams (ID) are a flexible tool to represent discrete stochastic optimization problems, including Markov Decision Process (MDP) and Partially Observable MDP as standard examples.
More precisely, given random variables considered as vertices of an acyclic digraph, a probabilistic graphical model defines a joint distribution via the conditional distributions of vertices given their parents.
In ID, the random variables are represented by a probabilistic graphical model whose vertices are partitioned into three types : chance, decision and utility vertices.
The user chooses the distribution of the decision vertices conditionally to their parents in order to maximize the expected utility.
Leveraging the notion of rooted junction tree, we present a mixed integer linear formulation for solving an ID, as well as valid inequalities, which lead to a computationally efficient algorithm.
We also show that the linear relaxation yields an optimal integer solution for instances that can be solved by the ``single policy update'', the default algorithm for addressing IDs.
\end{abstract}







\section{Introduction} 
\label{sec:introduction}

In this paper we want to address stochastic optimization problems with structured information and discrete decision variables, via mixed integer linear reformulations.
We start by recalling the framework of influence diagrams (more details can be found in~\cite[Chapter 23]{koller2009probabilistic}), and present the classical linear formulation for some special cases.


\subsection{The framework of parametrized influence diagram}
\label{subsec:framework}
Let $G=(V,E)$ be a directed graph, and, for each vertex $v$ in $V$, let $X_v$ be a random variable taking value in a finite state space $\calX_v$. For any $C \subset V,$ let $X_C$ denote $(X_v)_{v \in C}$ and $\calX_C$ be the cartesian product $\calX_C=\prod_{v \in C} \calX_v.$
We say that the distribution of the random vector $X_V$ factorizes as a \emph{directed graphical model} on $G$ if, for all $x_V \in \calX_V$, we have
\begin{equation}\label{eq:directedGraphicalModel}
	\bbP(X_V = x_V) = \prod_{v\in V} p_{v|\prt{v}}(x_v|x_{\prt{v}}),
\end{equation}
where $\prt{v}$ is the set of parents of $v$, that is, the set of vertices $u$ such that $(u,v)$ belongs to $E$, and $p_{v|\prt{v}}(x_v|x_{\prt{v}}) = \bbP(X_v = x_v|X_{\prt{v}} = x_{\prt{v}})$.
Further, given an arbitrary collection of conditional distributions
$\left\{p_{v|\prt{v}}\right\}_{v\in V}$,
Equation~\eqref{eq:directedGraphicalModel} uniquely defines a probability distribution on $\calX_V$.

Let $(\Va, \Vc, \Vl)$
be a partition of $V$ where $\Vc$ is the set of \emph{chances vertices}, $\Va$ is the set of \emph{decision vertices}, and $\Vl$ is the set of \emph{utility vertices} (the ones with no descendants). For ease of notation we denote $\Vs = \Vc \cup \Vl$.
Letters $\mathrm{a}$, $\mathrm{r}$, and $\mathrm{s}$ respectively stand for action, reward, and state in $\Va$, $\Vr$, and $\Vs$. 
We say that $G=(\Vs, \Va, E)$ is an \emph{Influence Diagram} (ID).
Consider a set of conditional distributions $\mathfrak{p} = \left\{ p_{v|\prt{v}}\right\}_{v\in\Vc \cup \Vl}$, and a collection of \emph{reward functions} $r = \{r_v\}_{v\in \Vl}$, with $r_v : \calX_v \rightarrow \bbR$.
Then we call $(G,\calX_V,\mathfrak{p},r)$ a \emph{Parametrized Influence Diagram} (PID).
We will sometimes refer to the parameters $(\calX_V,\mathfrak{p},r)$ by $\rho$ for conciseness.

Let $\Delta_v$ denote the set of conditional distributions $\delta_{v|\prt{v}}$ on $\calX_v$ given $\calX_{\prt{v}}$.
Given the set of conditional distributions $\mathfrak{p}
$,
a \emph{policy} $\delta$ in $\Delta = \prod_{v \in \Va} \Delta_v$,
uniquely defines a distribution $\bbP_\delta$ on $\calX_V$ through
\begin{equation}\label{eq:probabilityDistributionGivenPolicy}
	\bbP_{\delta}(X_V = x_V) = \prod_{v\in \Vs }p_{v|\prt{v}}(x_v|x_{\prt{v}}) \prod_{v\in \Va} \delta_{v|\prt{v}}(x_v|x_{\prt{v}}).
\end{equation}
Let $\bbE_{\delta}$ denote the corresponding expectation.
The \emph{Maximum Expected Utility} (MEU) problem associated to the
PID $(G,\calX_V,\mathfrak{p},r)$ is the maximization problem
\begin{equation}\label{pb:LIMID}
	\max_{\delta \in \Delta} \quad \bbE_{\delta}\Bgp{\sum_{v\in \Vl} r_v(X_v)} .
\end{equation}
A \emph{deterministic policy} $\delta \in \detpol \subset \Delta$, is such that for every $v \in \Va$, and any $x_v,x_{\prt{v}} \in \calX_{v} \times \calX_{\prt{v}}$, $\delta_{v|\prt{v}}(x_v | x_{\prt{v}})$ is a Dirac measure.
It is well known that there always exists an optimal solution to MEU~\eqref{pb:LIMID} that is deterministic (see \eg \citep[Lemma C.1]{liu2014reasoning} for a proof).

We conclude this section with some classical examples of IDs, shown in~\Cref{fig:LIMIDexamples}.
\begin{ex}
\label{ex:POMDP}
    Consider a maintenance problem where at time $t$ a machine is in state $s_t$.
    The action $a_t$ taken by the decision maker according to the current state is typically maintaining it (which is costly) or not (which increases the probability of failure).
    State and decision together lead to a new (random) state $s_{t+1}$, and the triple $(s_{t},a_t,s_{t+1})$ induces a reward $r_t$.
    This is an example of a Markov decision process (MDP) which is probably the simplest ID,
    represented in \Cref{fig:MDP_ex}.

    In practice, the actual state $s_t$ of the machine is often not known, but we only have some observation $o_t$ carrying partial information about the state,
    which leads to a more complex ID known as a partially observed Markov decision process (POMDP).
    In theory, an optimal decision should be taken knowing all past observations and decisions (which is the \emph{perfect recall} case).
    However, this would lead to policies living in spaces of exponentially large dimension and untractable MEU problems.
    It is thus common to restrict the decision $a_t$ to be made only based on observation $o_t$, 
     as illustrated in~\Cref{fig:POMDP_ex}.
\end{ex}
\begin{ex}
\label{ex:Chess}
    Consider two chess players : Bob and Alice.
They are used to play chess and for each game they bet a symbolic coin.
However, they can refuse to play.
Suppose that Alice wants to play chess every day.
On the day $t$, she has a current confidence level $s_t$.
The day of the game, her current mental fitness is denoted $o_t$.
When Bob meets with Alice, he makes the decision to play depending on her demeanor, denoted $u_t$. 
Then Bob can accept or decline the challenge, and his decision is denoted $a_t$.
Let $v_t$ denote the winner (getting a reward $r_t$). If Bob declines the challenge, there is no winner and no reward. Then, Alice's next confidence level is affected by the result of the game and her previous confidence level.
This stochastic decision problem can be modeled by an influence diagram as shown in \Cref{fig:example_chess}.

\end{ex}

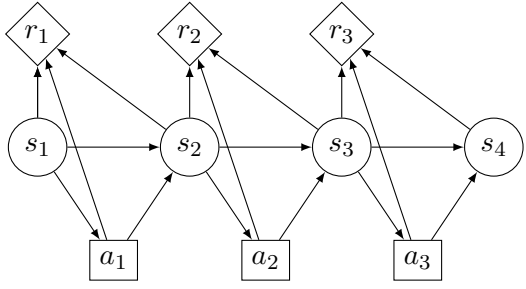
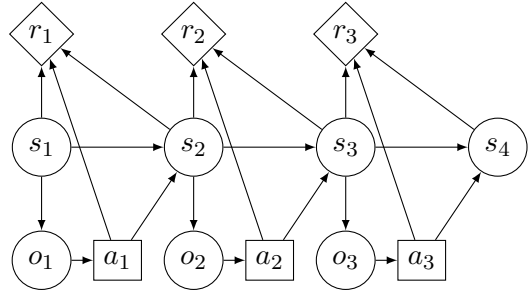
\begin{figure}
\begin{subfigure}{0.43\textwidth}
\centering
\begin{tikzpicture}

\def\h{1.5}

\node[sta] (s1) at (0,2*\h) {$s_1$};
\node[uti] (c1) at (0,3*\h) {$r_1$};
\node[act] (a1) at (1,1*\h) {$a_1$};

\node[sta] (s2) at (2,2*\h) {$s_2$};
\node[uti] (c2) at (2,3*\h) {$r_2$};
\node[act] (a2) at (3,1*\h) {$a_2$};

\node[sta] (s3) at (4,2*\h) {$s_3$};
\node[uti] (c3) at (4,3*\h) {$r_3$};
\node[act] (a3) at (5,1*\h) {$a_3$};

\node[sta] (s4) at (6,2*\h) {$s_4$};

\node (tar) at (3,0.6*\h) {};

\draw[arc] (s1) -- (s2);
\draw[arc] (s1) -- (c1);
\draw[arc] (s2) -- (c1);
\draw[arc] (a1) -- (c1);
\draw[arc] (s1) -- (a1);
\draw[arc] (a1) -- (s2);
\draw[arc] (s1) -- (c1);

\draw[arc] (s2) -- (c2);
\draw[arc] (s2) -- (a2);
\draw[arc] (s3) -- (c2);
\draw[arc] (a2) -- (c2);
\draw[arc] (a2) -- (s3);
\draw[arc] (s2) -- (s3);

\draw[arc] (s3) -- (c3);
\draw[arc] (s4) -- (c3);
\draw[arc] (s3) -- (a3);
\draw[arc] (a3) -- (c3);
\draw[arc] (a3) -- (s4);

\draw[arc] (s3) -- (s4);
\end{tikzpicture}
\caption{A Markov decision process (MDP)}
\label{fig:MDP_ex}
\end{subfigure}
~\hspace{1cm}
\begin{subfigure}{0.45\textwidth}
\centering
\begin{tikzpicture}

\def\h{1.5}

\node[sta] (s1) at (0,2*\h) {$s_1$};
\node[sta] (o1) at (0,1*\h) {$o_1$};
\node[uti] (c1) at (0,3*\h) {$r_1$};
\node[act] (a1) at (1,1*\h) {$a_1$};

\node[sta] (s2) at (2,2*\h) {$s_2$};
\node[sta] (o2) at (2,1*\h) {$o_2$};
\node[uti] (c2) at (2,3*\h) {$r_2$};
\node[act] (a2) at (3,1*\h) {$a_2$};

\node[sta] (s3) at (4,2*\h) {$s_3$};
\node[sta] (o3) at (4,1*\h) {$o_3$};
\node[uti] (c3) at (4,3*\h) {$r_3$};
\node[act] (a3) at (5,1*\h) {$a_3$};

\node[sta] (s4) at (6,2*\h) {$s_4$};

\node (tar) at (3,0.6*\h) {};

\draw[arc] (s1) -- (s2);
\draw[arc] (s1) -- (o1);
\draw[arc] (a1) -- (c1);
\draw[arc] (o1) -- (a1);
\draw[arc] (a1) -- (s2);
\draw[arc] (s1) -- (c1);
\draw[arc] (s2) -- (c1);

\draw[arc] (s2) -- (o2);
\draw[arc] (o2) -- (a2);
\draw[arc] (s2) -- (c2);
\draw[arc] (a2) -- (c2);
\draw[arc] (a2) -- (s3);
\draw[arc] (s2) -- (s3);

\draw[arc] (s3) -- (c2);
\draw[arc] (s3) -- (c3);
\draw[arc] (s3) -- (o3);
\draw[arc] (o3) -- (a3);
\draw[arc] (a3) -- (c3);
\draw[arc] (a3) -- (s4);
\draw[arc] (s4) -- (c3);

\draw[arc] (s3) -- (s4);

\end{tikzpicture}
\caption{A Partially Observed Markov Decision Process (POMDP) with limited memory}
\label{fig:POMDP_ex}
\end{subfigure}
\caption{ID examples, where we represent chance vertices ($\Vs$) in circles, decision vertices ($\Va$) in rectangles, and utility vertices ($\Vl$) in diamonds.}
\label{fig:LIMIDexamples}
\end{figure}


\begin{figure}[!ht]
    \begin{center}
        \begin{tikzpicture}
        \def\l{1.5}
        \def\h{1.2}

        \node[sta] (s1) at (-0.5*\l,2*\h) {$s_1$};
        \node[sta] (o1) at (0*\l,1*\h) {$o_1$};
        \node[sta] (u1) at (0*\l,0*\h) {$u_1$};
        \node[act] (a1) at (1*\l,0*\h) {$a_1$};
        \node[sta] (v1) at (1*\l,1*\h) {$v_1$};
        \node[uti] (r1) at (1.5*\l,-0.5*\h) {$r_1$};

        \node[sta] (s2) at (1.5*\l,2*\h) {$s_2$};
        \node[sta] (o2) at (2*\l,1*\h) {$o_2$};
        \node[sta] (u2) at (2*\l,0*\h) {$u_2$};
        \node[act] (a2) at (3*\l,0*\h) {$a_2$};
        \node[sta] (v2) at (3*\l,1*\h) {$v_2$};
        \node[uti] (r2) at (3.5*\l,-0.5*\h) {$r_2$};

        \node[sta] (s3) at (3.5*\l,2*\h) {$s_3$};
        \node[sta] (o3) at (4*\l,1*\h) {$o_3$};
        \node[sta] (u3) at (4*\l,0*\h) {$u_3$};
        \node[act] (a3) at (5*\l,0*\h) {$a_3$};
        \node[sta] (v3) at (5*\l,1*\h) {$v_3$};
        \node[uti] (r3) at (5.5*\l,-0.5*\h) {$r_3$};

        \node[sta] (s4) at (5.5*\l,2*\h) {$s_4$};

        \node (tar) at (3,0.6*\h) {};

        \draw[arc] (s1) -- (s2);
        \draw[arc] (s2) -- (s3);
        \draw[arc] (s3) -- (s4);

        \draw[arc] (s1) -- (o1);
        \draw[arc] (s2) -- (o2);
        \draw[arc] (s3) -- (o3);

        \draw[arc] (o1) -- (v1);
        \draw[arc] (o2) -- (v2);
        \draw[arc] (o3) -- (v3);

        \draw[arc] (o1) -- (u1);
        \draw[arc] (o2) -- (u2);
        \draw[arc] (o3) -- (u3);

        \draw[arc] (u1) -- (a1);
        \draw[arc] (u2) -- (a2);
        \draw[arc] (u3) -- (a3);

        \draw[arc] (v1) -- (s2);
        \draw[arc] (v2) -- (s3);
        \draw[arc] (v3) -- (s4);

        \draw[arc] (a1) -- (v1);
        \draw[arc] (a2) -- (v2);
        \draw[arc] (a3) -- (v3);

        \draw[arc] (v1) to [bend left=20] (r1);
        \draw[arc] (v2) to [bend left=20] (r2);
        \draw[arc] (v3) to [bend left=20] (r3);

        \draw[arc] (s1) -- (s2);

        \end{tikzpicture}
    \end{center}
    \caption{Bob and Alice chessgame}
    \label{fig:example_chess}
\end{figure}
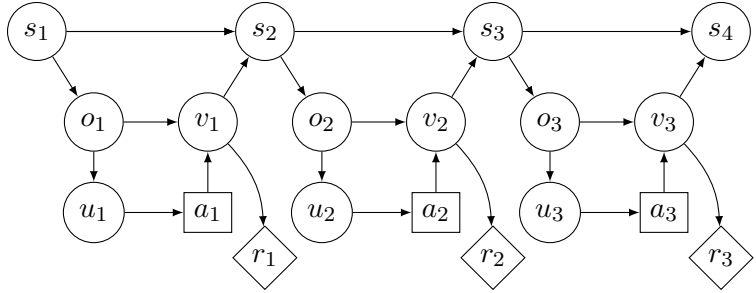

\subsection{Solving MDP through linear programs}


We recall here a well known linear programming formulation for MDP (see \eg \cite{puterman2014markov}), which is a special case of the Mixed Integer Linear Program (MILP) formulation introduced in the paper.
We denote by $p(s'|s,a)$ the probability of transiting from state $s$ to state $s'$ if action $a$ is taken, and $r(s,a,s')$ the reward associated to this transition.
For $t \in [T]:=\{1, \ldots, T\}$, let $\mu_s^t$ represent the probability of being in state $s$ at time $t$, let $\mu_{sa}^t$ represent the probability of being in state $s$ and taking action $a$ at time $t$, and let $\mu_{sas'}^t$ represent the probability of being in state $s$ and taking action $a$ at time $t$, while transiting to state $s'$ at time $t+1$.
This leads to the following mixed integer linear program
\begin{subequations}\label{pb:MDP_LP}
    \begin{alignat}{2}
            \max_{\boldsymbol{\mu}} \quad &{} \sum_{t=1}^{T-1} \sum_{s, a,s'} \mu_{sas'}^t \, r(s, a,s') \label{eq:MDP_obj}\\
            \mathrm{s.t.} \quad & \displaystyle \mu_{sas'}^t =  p(s'|s,a)\,\mu_{sa}^{t}, & \forall t,s,s',\label{eq:MDP_state_dynamic}\\
            & \mu_{s_0}^0 = 1,  \label{eq:MDP_initial_state}\\
            & \displaystyle \sum_{s,a} \mu_{sas'}^t = \mu_{s'}^{t+1},  & \forall t, s', \label{eq:MDP_mu_consistency}\\
            & \displaystyle \sum_{s} \mu_{s}^t = 1,  & \forall t ,\label{eq:MDP_mu_normalization}\\
            & \mu_s^t, \mu_{sa}^t , \mu_{sas'}^t \in\{0,1\}, &  \forall t,a,s,s',\label{eq:MDP_mu_nonnegativity}
    \end{alignat}
\end{subequations}
where the objective~\eqref{eq:MDP_obj} is simply the expected reward, Constraints~\eqref{eq:MDP_state_dynamic} represent the state dynamics, Constraints~\eqref{eq:MDP_initial_state} set the initial state of the system to $s_0$, and Constraints~\eqref{eq:MDP_mu_consistency}-\eqref{eq:MDP_mu_nonnegativity} ensure that $\mu$ represent marginals laws of a joint distribution. Integrity constraints~\eqref{eq:MDP_mu_nonnegativity} ensure that the policy chosen is deterministic.
In the MDP case, we can drop these integrity constraints and still obtain an optimal solution. In \Cref{sec:MILP}, the integrity constraints will come out to be useful in the general case.


\subsection{Literature}
\label{sub:literature}

Influence diagrams were introduced by \citet{howard1984readings} \citep[see also][]{howard2005influence}
to model stochastic optimization problems using a probabilistic graphical model framework.
Originally, the decision makers were assumed to have \emph{perfect recall} \citep{shenoy1992valuation,Shachter86EvalutingID, jensen1994influence} of the past actions family.

 \citet{lauritzen2001representing} relaxed this assumption\footnote{These authors used the name \emph{limited memory influence diagrams} when relaxing the perfect recall assumption, but we follow the convention of \citet{koller2009probabilistic} who still call them \emph{influence diagrams} (ID).}
and provided a simple (coordinate descent) algorithm to find a good policy: the Single Policy Update (SPU) algorithm.
The same authors also introduced the notion of \emph{soluble} ID as a sufficient condition for SPU to converge to an optimal solution.
This notion has been generalized by \citet{KOLLER&MILCH} to obtain a necessary and sufficient condition.
In general, SPU only finds a locally optimal policy, and requires to perform exact inference, so that it is therefore limited by the \emph{treewidth}~\citep{chandrasekaran2008complexity}.
More recently, \citet{maua2011solving} and \citet{maua2016fast} have introduced a new algorithm, \emph{Multiple Policy Update}, which has both an exact and a heuristic version and relies on dominance to discard partial solutions. It can be interpreted as a generalization of SPU where several decisions are considered simultaneously.
Later on, \citet{khaled2013solving} 
proposed a similar approach, with a Branch-and-Bound flavor, while \citet{liu2014reasoning} introduced heuristics based on approximate variational inference.
Finally, \citet{maua2016equivalences} has recently shown that the problem of solving an ID can be polynomially transformed into a maximum a posteriori (MAP) problem, and hence can be solved using popular MAP solvers such as \texttt{toulbar2} \citep{hurley2016multi}. 

Finding an optimal policy for an ID has been shown to be \NP-hard even when restricted to IDs of treewidth non-greater than two, or to trees with binary variables \citet{maua2012complexity,maua2013complexity}.
Note that even obtaining an approximate solution is also \NP-hard \citet{maua2012complexity}.

\emph{Credal networks} are generalizations of probabilistic graphical models where the parameters of the model are not known exactly. MILP formulations for credal networks that could be applied to IDs have been introduced by \citet{de2007inference,de2012strategy}.
However, the number of variables they require is exponential in the pathwidth, which is non-smaller and can be arbitrarily larger than the width of the tree we are using (follows from \citep[Theorem 4]{scheffler1990linear}),
and the linear relaxation of their MILP is not as good as the one of the MILP we propose, and does not yield an integer solution on soluble IDs.
Our approach 
 can naturally be extended to credal networks.
\subsection{Contributions}
\label{sub:contributions}

The contributions of the paper are as follows.

\begin{itemize}
    \item We introduce a non-linear program and a mixed integer linear program for the MEU problem on influence diagrams.
    \item These mathematical programs rely on a variant of the concept of a \emph{strong junction tree} which we introduce and call a \emph{rooted junction tree}.
    We provide algorithms to build rooted junction trees that lead to ``good'' mathematical programs for influence diagrams.
    \item We introduce a particular form of valid cuts for the obtained mixed integer linear program. 
    These valid cuts leverage conditional independence properties in the influence diagram.
    We show that our cuts are the strongest ones in a certain sense.
    We believe that this idea of leveraging conditional independence to obtain valid cuts is fairly general and could be extended to other contexts.
    \item We establish a link between the linear relaxation of our MILP and the concept of \emph{soluble relaxation} previously introduced in the literature on influence diagrams.
    In fact, our relaxation provides a better bound than those relaxations.
    \item We provide two new characterizations of soluble influence diagrams.
    First, as the only influence diagrams that can be solved to optimality using the linear relaxation of our mixed integer linear program.
    Second, and more importantly, as the influence diagrams for which there exists a rooted junction tree such that the set of collections of moments of distributions that are induced by the different policies is convex.
    \item We illustrate our mathematical programs and their properties on some simple numerical examples.
\end{itemize}

\subsection{Organization of the paper} 
\label{sub:contributions}

In \Cref{sec:tools}, we recall some definitions for graphical models, that are used
to extend the notion of \emph{junction tree} to \emph{rooted junction tree} in \Cref{sec:rooted_junection_trees}.
With these tools, \Cref{sec:MILP} introduces a bilinear formulation that can be rewritten as a mixed integer linear programming (MILP) formulation to the MEU Problem~\eqref{pb:LIMID}.
In \Cref{sec:validCuts} we give efficient valid cuts
for the MILP formulation, and interpret them in terms of graph relaxations.
\Cref{sec:soluble_} studies the polynomial case of soluble ID, showing that the ID that can be solved to optimality
by SPU can be solved by (continuous) linear programming using our formulation.
Finally \Cref{sec:numerical} summarizes our numerical experiments.


\section{Tools from Probabilistic graphical model theory}
\label{sec:tools}

In this section we present notations and tools used in the following sections to refomulate the MEU Problem~\ref{pb:LIMID}.

\subsection{Graph notation}
\label{sub:graph_notation}

This section introduces our notations for graphs, which are for the most part the ones commonly used in the combinatorial optimization community \citep{schrijver2003combinatorial}.
A directed graph $G$ is a pair $(V,E)$ where $V$ is the set of vertices and $E \subseteq V^2$ the set of arcs.
We write $u \rightarrow v$ when $(u,v) \in E$. Let  $[k]:=\{1, \ldots, k\}$.
A \emph{path} is a sequence of vertices $v_1,\ldots,v_k$ such that $v_i \rightarrow v_{i+1}$, for any $i \in [k-1]$. 
A path between two vertices $u$ and $v$ is called a $u$-$v$ path.
We write
$u \rcarrow[G] v$ to denote the existence of a $u$-$v$ path in $G$, or simply $u \rcarrow v$ when $G$ is clear from context.
We write $u \rightleftharpoons v$ if there is an arc $u \rightarrow v$ or $v\rightarrow u$.
A \emph{trail} is a sequence of vertices $v_1,\ldots,v_k$ such that $v_i \rightleftharpoons v_{i+1},$ for all $i \in [k-1]$.

A \emph{parent} (resp.~\emph{child}) of a vertex $v$ is a vertex $u$ such that $(u,v)$ (resp.~$(v,u)$) belongs to~$E$; we denote by $\prt{v}$ the set of parents vertices (resp.~$\cld{v}$ the set of children vertices).

The \emph{family} of $v$, denoted by $\fa{v}$, is the set $\{v\}\cup\prt{v}$.
A vertex $u$ is an \emph{ascendant} (resp.~a \emph{descendant}) of $v$ if there exists a $u$-$v$~path. 
We denote respectively $\asc{v}$ and $\dsc{v}$ the set of ascendants and descendants of $v$.
Finally, let $\casc{v} = \{v\}\cup \asc{v}$, and $\cdsc{v} = \{v\} \cup \dsc{v}$.
For a set of vertices $C$, the parent set of $C$, again denoted by $\prt{C}$, is the set of vertices $u$ that are parents of a vertex $v\in C$.
We define similarly $\fa{C}$, $\cld{C}$, $\asc{C}$, and $\dsc{C}$.
Note that we sometimes indicate in subscript the graph according to which the parents, children, etc.,~are taken.
For instance, $\prt[G]{v}$ denotes the parents of $v$ in $G$.
We drop the subscript when the graph is clear from the context.




A \emph{cycle} is a path $v_1,\ldots,v_k$ such that $v_1 = v_k$.
A graph is \emph{connected} if there exists a path between any pair of vertices.
An undirected graph is a \emph{tree} if it is connected and has no cycles.
A directed graph is a \emph{directed tree} if its underlying undirected graph is a tree.
A \emph{rooted tree} is a directed tree such that all vertices have a common ascendant referred to as the \emph{root} of the tree\footnote{The probabilistic graphical model community sometimes calls a \emph{directed tree} what we call here a \emph{rooted tree}, and a \emph{polytree} what we call here a \emph{directed tree}.}.
In a rooted tree, all vertices but the root have exactly one parent.











\subsection{Directed graphical model}
\label{sub:directed_graphical_model}


In this paper, we manipulate several distributions on the same random variables. Given three random variables $X$, $Y$, $Z$, the notation
$ \big(X \indep Y | Z\big)_\mu $
stands for ``$X$ is independent from $Y$ given $Z$ according to $\mu$''. The parenthesis $(\cdot)_\mu$ are dropped when $\mu$ is clear from context.
The same notation is used for independence of events.

A well-known sufficient condition for a distribution to factorize as a probabilistic graphical model is that each vertex is independent from its non-descendants given its parents.
\begin{prop}\label{prop:factorizationDirectedGraph}\citep[Theorem 3.1, p.~62]{koller2009probabilistic}
Let $\bbP_{\mu}$ be a distribution on $\calX_V$.
Then $\bbP_{\mu}$ factorizes as a directed graphical model on $G$, that is
$$\bbP_{\mu}(X_V = x_V) = \prod_{v \in V} \bbP_{\mu}(X_v = x_v | X_{\prt{v}} = x_{\prt{v}}),$$
if and only if 
\begin{equation}\label{eq:indepFromDesc}
	 \left(X_v \indep X_{V \backslash \cdsc[G]{v}} | X_{\prt{v}}\right)_\mu \quad \text{for all $v$ in }V.
\end{equation}
\end{prop}

\subsection{Junction trees}
\label{sub:directed_graphical_models}

When dealing with the MEU Problem~\ref{pb:LIMID}, one needs to deal with distributions $\mu_V$ on $\calX_V$ that factorize as in \eqref{eq:probabilityDistributionGivenPolicy} for some policy $\delta$.
In theory, it suffices to consider distributions $\mu_V$ satisfying the conditional independences given by \Cref{eq:indepFromDesc} and such that $\bbP_{\mu}(X_v | X_{\prt{v}}) = p_{v|\prt{v}}$ for each vertex $v$ that is not a decision.
However, the joint distribution $\mu_V$ on all the variables is too large to be manipulated in practice as soon as $V$ is moderately large. 
In that case, it is handy to work with a \emph{vector of moments} $\tau = (\tau_C)_{C\in \calV}$, where $\calV \subseteq 2^V$, that is, a vector of distributions $\tau_C$ on subsets of variables $C$ of tractable size.
A vector of moment $(\tau_C)_{C \in \calV}$ derives from a distribution $\mu_V$ on $\calX_V$ if each moment $\tau_C \in [0,1]^{\calX_{C}}$ is the marginal of $\mu_V$, \ie $\tau_C(x_C)=\sum_{x_{V\backslash C} \in \calX_{V \backslash C}}\mu_V(x_C,x_{V\backslash C})$ for all $C$ in $\calV$ and $x_C$ in $\calX_C$. To keep notations light, we will write this type of equality more compactly as $\tau_C=\sum_{x_{V\backslash C}}\mu_V$.
We use the notation $\mu=(\mu_C)_{C \in \calV}$ for the vector of moments deriving from a distribution, and $\bbP_\mu$ or $\mu_V$ for the corresponding distribution on $\calX_V$.

A necessary condition for a vector of moments $(\tau_C)_{C \in \calV}$ to derive from a distribution is to be \emph{locally consistent}, that is to induce the same marginals on the intersections of pairs of elements of $\calV$, i.e., that for all $C_1,C_2 \in \cal V,$ we have 
$$\sum_{x_{C_1 \backslash C_2}} \tau_{C_1}= \sum_{x_{C_2 \backslash C_1}} \tau_{C_2},$$
where, as before, $\sum_{x_{C_1 \backslash C_2}} \tau_{C_1}$ is the vector $\big (\sum_{x_{C_1 \backslash C_2} \in \calX_{C_1 \backslash C_2}} \tau_{C_1}(x_{C_1 \backslash C_2}, x_{C_1 \cap C_2}) \big )_{x_{C_1 \cap C_2} \in \calX_{C_1 \cap C_2}}$.
It turns out that graphical model theory provides a condition on the choice of $\calV$ together with the choice of local consistency constraints which are sufficient for $(\tau_C)_{C \in \calV}$ to derive from a distribution on $\calX_V$. This is done via the definition of a junction tree. 
Let $\Gcl = (\Vcl,\Acl)$ be an undirected graph associated with $G= (V,E)$ with $\Vcl \subseteq 2^V$, and such that there is a mapping $v \mapsto C_v$ from $V$ to $\Vcl$ satisfying that  $\fa{v} \subseteq C_v$.
If $\calG$ is a tree, and satisfies the \emph{running intersection property}, i.e., that given two vertices $C_1$ and $C_2$ in $\calV$, any vertex $C$ on the unique undirected path from $C_1$ to $C_2$ in $\calG$ satisfies $C_1 \cap C_2 \subset C$, then $\calG$ is called a \emph{junction tree} of $G$.
See \Cref{fig:jt-rjt-example} for an illustration of this notion.
Given a junction tree $\calG$, its associated \emph{marginal polytope} $\calL^0_\calG$ is defined as follows
\begin{equation}\label{eq:localPolytope}
\calL^0_\calG = \left\{ (\tau_C)_{C \in \calV} \colon \left| 
\begin{array}{l}
\displaystyle\tau_C \geq 0 \quad \text{and} \quad \sum_{x_C}\tau_C(x_C) = 1 \quad \forall x_C \in \calX_C,
\: \forall C \in \calV,  
 \\
\text{ and }\quad \displaystyle\sum_{x_{C_1 \backslash C_2}} \tau_{C_1}
= \sum_{x_{C_2 \backslash C_1}} \tau_{C_2}, \quad
\forall \{C_1,C_2\} \in \calA, 
\end{array}
\right.\right\}
\end{equation}
Then $\tau=(\tau_C)_{C \in \calV}$ is a vector of moments deriving from a distribution $\mu_V$ on $\calX_V$ if and only if $\tau \in \calL^0_\calG$  \citep[Proposition 2.1]{wainwright2008graphical}.


\section{Rooted junction trees}
\label{sec:rooted_junection_trees}


To solve the MEU Problem~\eqref{pb:LIMID}, we work on vectors of moments $(\mu_C)_{C \in \calV}$ that correspond to the moments of distributions $\mu$ induced by policies $\delta \in \Delta$.
Hence, we are interested in vectors $\mu$ of moments such that $\mu_V$ factorizes as a directed graphical model on $G$.
Such vectors of moments necessarily satisfy a ``local'' version of the sufficient condition~\eqref{eq:indepFromDesc}, which is that for $\tau_C=\mu_C,$
\begin{equation}\label{eq:localIndepFromDesc}
	\big(X_v \indep X_{C \backslash \cdsc{v}} | X_{\prt{v}}\big)_{\tau_C} \quad \text{for all }C \in \calV, \text{ for all }v\in V\colon \fa{v}\subseteq C.
\end{equation}
Given a vector of moment $\tau_C$ in the local polytope of a junction tree $(\calV,\calA)$, satisfying \eqref{eq:localIndepFromDesc} is not a sufficient condition for $\tau_C$ to be the moments of a distribution $\mu_V$ that factorizes on $G$.
But it becomes a sufficient condition under the additional assumption that $(\calV,\calA)$ is a ``rooted junction tree'', a notion that we introduce in this section, and develop in more detail in~\Cref{app:RJT}.

\subsection{Definition and main properties}
\label{sub:definition_and_main_properties}



Let $\calG= (V,E)$ be a junction tree on $G=(\Vcl,\Acl)$ and $v\in V$ a vertex of $G$.
Then, thanks to the running intersection property, the subgraph $\calG_v$ of $\calG$ made of all nodes $C\in \Vcl$ containing $v$ is a tree.
Moreover, any orientation of the edges of $\calG$ that makes it a rooted tree, also makes $\calG_v$ a rooted tree, and we denote $C_v$ its root node.

\begin{de}
\label{def:rjt}
A \emph{rooted junction tree} (\rjt)
on $G= (V,E)$ is
a rooted tree with nodes in $2^V$, such that
\begin{itemize}
	\item[(i)] its underlying undirected graph $\Gcl = (\Vcl,\Acl)$ is a junction tree,
	\item[(ii)] for all $v\in V$, we have $\fa{v} \subseteq C_v$,
\end{itemize}
where $C_v$ is the \emph{root clique of $v$} defined as the root node of
the subgraph $\calG_v$ of $\calG$ induced by the nodes $C \in \Vcl$ containing $v$.

Let $\Gcl$ be an \rjt on $G$, and $v$ a vertex of $V$. Given $C \in \Gcl$, let $\offspring{C}=\{v\in V: C_v = C\}$, which we call the \emph{offspring} of $C,$ and let $\check{C}$ denote $C \backslash \offspring{C}$.
\end{de}
See Figure~\ref{fig:jt-rjt-example} for a graphical example of this notion.
Note that an \rjt always exists: Indeed, the cluster graph composed of a single vertex $C= V$ is an \rjt.
Algorithms to build interesting \rjt are provided in~\Cref{sub:building_a_rjt}.

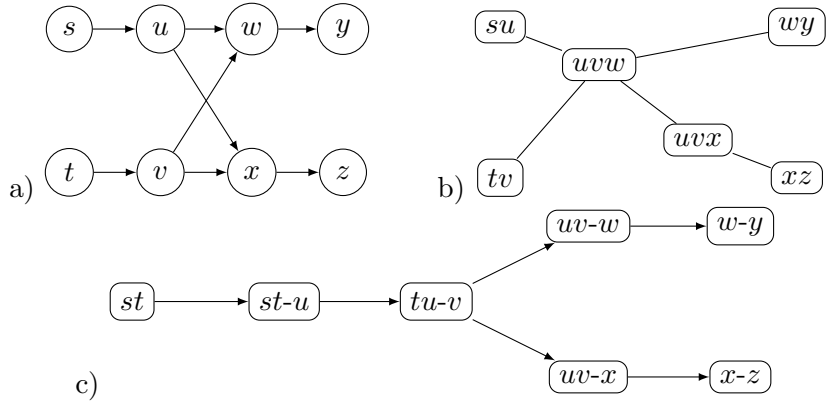
\begin{figure}
\begin{center}
a)
\begin{tikzpicture}
\def\l{1.2}
\def\h{1.9}
\node[sta] (s) at (0*\l,1*\h) {$s$};
\node[sta] (t) at (0*\l,0*\h) {$t$};
\node[sta] (u) at (1*\l,1*\h) {$u$};
\node[sta] (v) at (1*\l,0*\h) {$v$};
\node[sta] (w) at (2*\l,1*\h) {$w$};
\node[sta] (x) at (2*\l,0*\h) {$x$};
\node[sta] (y) at (3*\l,1*\h) {$y$};
\node[sta] (z) at (3*\l,0*\h) {$z$};
\draw[arc] (s) -- (u);
\draw[arc] (t) -- (v);
\draw[arc] (v) -- (w);
\draw[arc] (v) -- (x);
\draw[arc] (u) -- (w);
\draw[arc] (u) -- (x);
\draw[arc] (w) -- (y);
\draw[arc] (x) -- (z);
\end{tikzpicture}
\quad\quad b)
\begin{tikzpicture}
\def\l{1.3}
\def\h{1}
\node[draw,rounded corners] (C1) at (0*\l,2*\h) {$su$};
\node[draw,rounded corners] (C2) at (0*\l,0*\h) {$tv$};
\node[draw,rounded corners] (C3) at (1*\l,1.5*\h) {$uvw$};
\node[draw,rounded corners] (C4) at (2*\l,00.5*\h) {$uvx$};
\node[draw,rounded corners] (C5) at (3*\l,2*\h) {$wy$};
\node[draw,rounded corners] (C6) at (3*\l,0*\h) {$xz$};
\draw (C1) -- (C3);
\draw (C2) -- (C3);
\draw (C3) -- (C4);
\draw (C3) -- (C5);
\draw (C4) -- (C6);
\end{tikzpicture}
\\
\enskip c)
\begin{tikzpicture}
\def\l{2}
\def\h{1}
\node[draw, rounded corners] (C1) at (-2*\l,1*\h) {$st$};
\node[draw, rounded corners] (C2) at (-1*\l,1*\h) {$st$-$u$};
\node[draw, rounded corners] (C3) at (0*\l,1*\h) {$tu$-$v$};
\node[draw, rounded corners] (C4) at (1*\l,0*\h) {$uv$-$x$};
\node[draw, rounded corners] (C5) at (1*\l,2*\h) {$uv$-$w$};
\node[draw, rounded corners] (C6) at (2*\l,2*\h) {$w$-$y$};
\node[draw, rounded corners] (C7) at (2*\l,0*\h) {$x$-$z$};
\draw[arc] (C1) -- (C2);
\draw[arc] (C2) -- (C3);
\draw[arc] (C3) -- (C5);
\draw[arc] (C3) -- (C4);
\draw[arc] (C5) -- (C6);
\draw[arc] (C4) -- (C7);
\end{tikzpicture}
\end{center}
\caption{a)~A directed graph $G$, b)~a junction tree on $G$, and c)~a rooted junction tree on $G$, where, for each cluster $C$, we indicate on the left part of the labels the vertices of $C\backslash \offspring{C}$, and on the right part the vertices of $\offspring{C}$.}
\label{fig:jt-rjt-example}
\end{figure}

Theorem \ref{theo:rjtFactorization}, which is a natural generalization of the well-known Proposition~\ref{prop:factorizationDirectedGraph}, ensures that given a vector of moments on an \rjt that satisfies local independences, we can construct a distribution on the initial directed graphical model which admits these moments as marginals.


\begin{theo}\label{theo:rjtFactorization}
Let $\mu$ be a vector of moments in the local polytope of an \rjt $\calG$ on $G= (V,E)$.
Suppose that for each vertex $v$, according to $\mu_{C_v}$, the variable $X_v$ is independent from its non-descendants in $G$ that are in $C_v$, conditionally to its parents.
Then there exists a distribution $\bbP_\mu$ on $\calX_V$ factorizing on $G$ with moments $\mu$.
\end{theo}

\begin{rem}\label{rem:smallOffspring}
By adding nodes to an \rjt, we can always turn it into an \rjt satisfying $\offspring{C_v} = \{v\}$ for each vertex $v$ in $\Va$. Indeed, suppose that $\offspring{C} = \{v_1,\ldots,v_k\}$, where $v_1,\ldots,v_k$ are given along a topological order.
It suffices to replace the node $C$ by $C_1 \rightarrow C_2  \rightarrow \dots \rightarrow C_k$, where $C_i = C\backslash \{v_{i+1},\ldots,v_k\}$. Note that for such RJTs we have $\check{C}_v=C_{v} \backslash \{v\}.$
\end{rem}

\begin{rem}\label{rem:strongJunctionTree}
\citet[beginning of Section 4]{jensen1994influence} introduces a similar notion of \emph{strong junction tree}.
It relies on the notion of \emph{elimination ordering} for a given \emph{influence diagram} with perfect recall.
The main difference is that a strong junction tree is a notion on an influence diagram, where the set of decision vertices and their orders play a role, when \rjts rely on the underlying digraph.
The notion of strong junction tree is obtained by replacing (ii) in the definition of an \rjt by: ``given an elimination ordering, if $(C_u,C_v)$ is an arc, there exists an ordering of $C_v$ that respects the elimination ordering such that $C_u \cap C_v$ is before $C_v \backslash C_u$ in that ordering''.
An \rjt is a strong junction tree.
Conversely, a strong junction tree is not necessarly an \rjt. Indeed, \citet[Figure 4]{jensen1994influence} shows an example of strong junction where there is $v \in V$ such that $\fa{v} \subsetneq C_v$.
As strong junction trees is a notion on influence diagram and not on graphs, Theorem~\ref{theo:rjtFactorization} has no natural generalization for strong junction trees.
\end{rem}



\subsection{Building an \rjt}
\label{sub:building_a_rjt}

Although $(\{V\},\emptyset)$ is a rooted junction tree, the concept has only practical interest if it is possible to construct {\rjt}s with small cluster nodes.
In that respect, note that any \rjt must satisfy, for all $u,v \in V,$ the implication
\begin{equation}
\label{eq:rjt_prop}
\left .\begin{aligned}
\exists w \in V\: \st &C_v \rcarrow C_w\: \text{and} \: u \in \fa{w}\:\\
\text{and}\qquad &C_u \rcarrow C_v
\end{aligned}
\right \} \Rightarrow u \in C_v,
\end{equation}
where $C \rcarrow C'$ denotes the existence of a $C$-$C'$ path in the \rjt $\calG$ considered. 
This notation will be used throughout this section.
Indeed, since $u \in C_u$ and $\fa{w} \subset C_w$ by definition, and since $C_u \rcarrow C_v \rcarrow C_w$, the running intersection property implies $u \in C_v$.
This motivates Algorithm~\ref{alg:buildRJT}, a simple RJT construction algorithm which propagates iteratively elements present in each cluster node to their parent cluster node.
Let $\preceq$ be an arbitrary topological order on $G$, and $\max_{\preceq}C$ denote the maximum of $C$ for the topological order $\preceq$.
The algorithm maintains a set $C'_v$
for each vertex $v$, which coincide at the end of the algorithm with the nodes $C_v$ in the \rjt produced.
We denote by $\D{v}'$ is the set $C'_v \backslash \{v\}$.
As we will show, Algorithm~\ref{alg:buildRJT} produces an \rjt $\calG =(\calV,\calA)$ which is minimal for $\preceq$, in the sense that it satisfies a converse of \eqref{eq:rjt_prop}.

\begin{algorithm}[H]
\caption{Create an \rjt given a topological order}
\label{alg:buildRJT}
\begin{algorithmic}[1]
\STATE \textbf{Input} $G = (V,E)$ and a topological order $\preceq$ on G
\STATE \textbf{Initialize} $C'_v = \emptyset$ for all $v \in V$ and $\calA' = \emptyset$
\FOR{each node $v$ of $V$ taken in reverse topological order $\preceq$}
\STATE $C'_v \leftarrow  \fa{v} \cup \bigcup_{w: (v,w) \in \calA'}\D{w} $ \label{step:Cvdef}
\IF{$\D{v}' \neq \emptyset$}
\STATE $u \leftarrow \max_{\preceq}\big(\D{v}\big)$ \label{step:arc}
\STATE $\calA' \leftarrow \calA' \cup (u,v)$
\ENDIF
\ENDFOR
\STATE $\calA \leftarrow \{(C'_u,C'_v) \mid (u,v) \in \calA'\}$
\STATE \textbf{Return} $\calG = \big  ((C'_v)_{v \in V} , \calA \big )$
\end{algorithmic}
\end{algorithm}



\begin{rem}
Algorithm~\ref{alg:buildRJT} takes as input a topological order on $G$.
For a practical use, we recommend to use Algorithm~\ref{alg:build_smallest_RJT2} in Appendix~\ref{sec:algorithm_rjt}, which builds simultaneously the \rjt and a ``good'' topological order.
\end{rem}
For instance, for any topological order on the graph of the chess example of Figure~\ref{fig:example_chess}, Algorithm~\ref{alg:buildRJT} produces the \rjt illustrated on Figure~\ref{ex:chess_rjt}. 

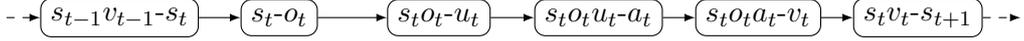
\begin{figure}
\begin{center}
\begin{tikzpicture}
\def\l{2.1}
\node[draw,rounded corners] (s) at (0*\l,0) {$s_{t-1}v_{t-1}$-$s_t$};
\node[draw,rounded corners] (o) at (1*\l,0) {$s_t$-$o_t$};
\node[draw,rounded corners] (u) at (2*\l,0) {$s_to_t$-$u_t$};
\node[draw,rounded corners] (a) at (3*\l,0) {$s_to_tu_t$-$a_t$};
\node[draw,rounded corners] (v) at (4*\l,0) {$s_to_ta_t$-$v_t$};
\node[draw,rounded corners] (t) at (5*\l,0) {$s_tv_t$-$s_{t+1}$};
\draw[arc] (s) -- (o);
\draw[arc] (o) -- (u);
\draw[arc] (u) -- (a);
\draw[arc] (a) -- (v);
\draw[arc] (v) -- (t);
\draw[arc,<-,dashed] (s.west) -- ++(-0.5,0);
\draw[arc,->,dashed] (t.east) -- ++(0.5,0);

\end{tikzpicture}
\end{center}
\caption{Rooted junction tree produced by Algorithm~\ref{alg:buildRJT} on the example of Figure~\ref{fig:example_chess}. The offspring of a node is to the right of symbol -.}
\label{ex:chess_rjt}
\end{figure}
The following proposition, whose proof can be found in \Cref{app:RJT} shows that Algorithm~\ref{alg:buildRJT} builds the minimal \rjt .
\begin{prop}
\label{prop:algcorrect}
Algorithm~\ref{alg:buildRJT} produces an \rjt such that the root node $C_v$ of $v$ is $C'_v$, satisfying $\offspring{C_v} = \{v\}$, that admits $\preceq$ as a topological order, and such that $(u \in C_v) \Rightarrow (u \preceq v)$. Moreover, its cluster nodes are minimal in the sense that
\begin{equation}
\label{eq:minimality_rjt}
u \in C_v \Rightarrow \left \{ \begin{aligned}
\exists w \in V\: \st &C_v \rcarrow C_w\:\: \text{and} \:\: u \in \fa{w},\:\\
&C_u \rcarrow C_v.
\end{aligned}
\right.
\end{equation}
\end{prop}

\section{MILP formulation for influence diagrams}
\label{sec:MILP}

Given that \Cref{alg:buildRJT} produces an \rjt such that $\offspring{C_v} = \{v\}$ for all $v \in V,$ we will assume in the rest of the paper that all the RJTs considered satisfy this property.
As noted in Remark~\ref{rem:smallOffspring}, any \rjt can be turned into an \rjt satisfying this property by adding more nodes.
In the rest of the paper, we work with the following variant of the local polytope $\calL_\calG^0$ defined in Equation~\eqref{eq:localPolytope}
$$ \calL_\calG = \bigg\{(\mu_{C_v},\mu_{\D{v}})_{v \in V} \colon (\mu_{C_v})_{v \in V} \in \calL_\calG^0 \text{ and } \mu_{\D{v}} = \sum_{x_v} \mu_{C_v} \bigg\},$$
where moments $\mu_{\D{v}}$ have been introduced. This is for convenience, and all the results could have been written using  $\calL_\calG^0$.


On graphical models, the inference problem, which is hard in general, becomes easy on junction trees.
Since problem~\eqref{pb:LIMID} is $\NP$-hard even when restricted to graphs of treewidth 2 \citep{maua2012solving},
unless $P = \NP$, the situation is strictly worse for the MEU problem associated with influence diagrams.
However, we will see in this section that, given a rooted junction tree, we can obtain 
mathematical programs to solve the \meu problem~\ref{pb:LIMID} with a tractable number of variables and constraints provided that cliques are of reasonable size. 
We first obtain an NLP formulation in Section~\ref{sub:an_exact_nlp_formulation}, and then linearize it into an exact mixed integer linear program (MILP) in Section~\ref{sub:milp_formulation}.


\subsection{An exact Non Linear Program formulation}
\label{sub:an_exact_nlp_formulation}
Consider a \emph{Parameterized Influence Diagram} (PID) encoded as the quadruple $(G,\calX,\mathfrak{p},r)$, where $G=(V,E)$ is
a graph with set of vertices $V$ partitioned into $(\Va, \Vs)$, with
$\calX = \prod_{v\in V} \calX_v$ the support of the vector of random variables attached to all vertices of $G$,
$\mathfrak{p}= \{p_{v|\prt{v}}\}_{v\in \Vs}$ is the collection of fixed and assumed known conditional probabilities,
and $r = \{r_v\}_{v\in\Vl}$ is the collection of reward functions\footnote{we remind the reader that $\Vl$ is the set of utility vertices as introduced in \Cref{subsec:framework}}
$r_v:\calX_v \rightarrow \bbR$ which we will also view as vectors $r_v\in \bbR^{|\calX_v|}.$


For $(G,\calX,\mathfrak{p},r)$ a given PID, and $\calG$ a given \rjt,
we introduce the following polytope
\begin{equation}\label{eq:Pbar}
    \overline{\calP}(G, \calX, \mathfrak{p}, \calG) =
    \big\{\mu \in \calL_\calG \colon
    \mu_{C_v} = \mu_{\D{v}} \, p_{v|\prt{v}}
     \text{ for all } v \in \Vs
    \big\} .
\end{equation}
where the equality $\mu_{C_v} = \mu_{\D{v}} \, p_{v|\prt{v}}$ should be understood functionally, i.e., meaning that  $\mu_{C_v}(x_{C_v}) = \mu_{\D{v}}(x_{\D{v}}) \, p_{v|\prt{v}}(x_v|x_{\prt{v}}), \: \forall x_{C_v} \in \calX_{C_v};$ we will use such functional (in)equalities throughout the paper.
We omit the dependence of $\overline{\calP}$ in $(G, \calX, \mathfrak{p}, \calG)$ when the context is clear.
Consider the following Non Linear Program (NLP)
\begin{subequations}\label{pb:NLP}
\begin{alignat}{2}
\max_{\mu,\delta}  \enskip & \sum_{v \in \Vl} \langle r_v , \mu_v \rangle & \quad &\\
\mathrm{s.t.} \enskip
& \mu \in \overline{\calP}(G, \calX, \mathfrak{p}, \calG) \label{eq:NLP_localPolytope}\\
&\delta \in \Delta \\
 & \mu_{C_v} = \delta_{v|\prt{v}} \, \mu_{\D{v}}, &&\forall v \in \Va, \label{eq:NLP_fact_Va}
\end{alignat}
\end{subequations}
\noindent where the inner product notation $\langle r_v ,
\mu_v \rangle$ stands for $\sum_{x_v} \mu_v(x_v) r_v(x_v)$.
Note that the constraints $\delta \in \Delta$ are implied by the other ones.




\begin{theo}\label{theo:NLPexact}
The (NLP) Problems~\eqref{pb:NLP} and~\eqref{eq:NLP_calS} are equivalent to the MEU Problem~\eqref{pb:LIMID},
in the sense that they have the same value and that, if $(\mu, \delta)$ is a feasible solution for Problem~\eqref{pb:NLP}, then $\delta$ defines an admissible policy for Problem~\eqref{pb:LIMID}, and $\mu$ characterizes the moments of the distribution induced by $\delta$.
\end{theo}
\begin{proof}
If $(\mu,\delta)$ is a solution of~\eqref{pb:NLP}, then $\mu$ is is a solution of~\eqref{eq:NLP_calS}, and conversely, if $\mu$ is a solution of~\eqref{eq:NLP_calS}, by definition of $\calS(G)$, there exists $\delta$ such that $(\mu,\delta)$ is a solution of~\eqref{pb:NLP}, which gives the equivalence between~\eqref{pb:NLP} and ~\eqref{eq:NLP_calS}.

Let now $(\mu,\delta)$ be an admissible solution of Problem~\eqref{pb:NLP}. Then $\delta$ is an admissible solution of the MEU problem. We now prove that $\mu$ corresponds to the moments of the distribution $\bbP_{\delta}$ induced by $\delta$, from which we can deduce that $\bbE_{\delta}\Bp{\sum_{v\in \Vl} r_v(X_v)} = \sum_{v \in \Vl} \langle r_v , \mu_v \rangle$.
Note that, if $A$, $P$, and $D$ are disjoint subsets of $V$, 
$\mu$ is a distribution on $\calX_V$, $\mu_{A\cup P \cup D}$ is the distribution induced by $\mu$ on $\calX_{A \cup P \cup D}$, and $p_{D|P}$ is a conditional distribution of $D$ given $P$,
then
\begin{equation}
\label{eq:conditionalDistributionGivesIndependence}
    \mu_{A \cup P \cup D}=\mu_{A \cup P} \, p_{D|P} \qquad \Longrightarrow \qquad X_D \indep X_A \mid X_P,
\end{equation}
where the independence is according to $\mu$.
By~\eqref{eq:conditionalDistributionGivesIndependence}, we have that the vector $\mu$ satisfies the conditions of Theorem~\ref{theo:rjtFactorization}, and hence corresponds to a distribution $\bbP_\mu$ that factorizes on $G$. Furthermore, 
constraint~\eqref{eq:Pbar}
ensures that $\bbP_\mu(X_v|X_{\prt{v}})= p_{v|\prt{v}}$ for all $v \in \Vs$, which yields the result.
Conversely, let $\delta$ be an admissible solution of the MEU Problem~\eqref{pb:LIMID}, 
and $\mu$ be the vector of moments induced by $\bbP_{\delta}$. We have $\mu_{C_v} = \mu_{\D{v}}p_{v|\prt{v}}$ for $v$ in $\Vs$ 
and $\mu_{C_v} = \mu_{\D{v}}\delta_{v|\prt{v}}$ for $v$ in $\Va$, 
and $(\mu,\delta)$ is a solution of~\eqref{pb:NLP}. 
Furthermore, $\bbE_{\delta}\Bp{\sum_{v\in \Vl} r_v(X_v)} = \sum_{v \in \Vl} \langle r_v , \mu_v \rangle$, and \eqref{pb:NLP} is equivalent to the MEU Problem~\eqref{pb:LIMID}.
\end{proof}


By introducing the following set of moments
\begin{equation}\label{eq:defQcalG}
\calS(G) = \big\{
\mu \in \overline{\calP}  \colon \exists \delta \in \Delta, \mu_{C_v} = \mu_{\D{v}} \delta_{v|\prt[G]{v}} \text{ for all $v$ in $\Va$}
\big\},
\end{equation}
we can reformulate the Problem~\eqref{pb:NLP} more concisely as
\begin{equation}\label{eq:NLP_calS}
    \max_{\mu \in \calS(G)}\sum_{v \in \Vl} \langle r_v, \mu_v\rangle.
\end{equation}
$\calS(G)$ is the set of moments corresponding to distributions induced by feasible policies: $\mu $ is in $ \calS(G)$ if there exists $\delta$ in $\Delta$ such that $\mu_{C_v}(x_{C_v}) = \bbP_{\delta}(X_{C_v} = x_{C_v})$ for all $v$ 
and $x_{C_v}$.
It is non-convex in general as shown by the examples in the proof of Theorem~\ref{theo:validEqualSoluble}.
However, we show in Section~\ref{sec:soluble_} that $\calS(G)$ is a polytope if $G$ is soluble,
a property identifying ``easy'' IDs.

\subsection{MILP formulation}
\label{sub:milp_formulation}
The NLP \eqref{pb:NLP} is hard to solve due to the non-linear constraints~\eqref{eq:NLP_fact_Va}.
But by Theorem~\ref{theo:NLPexact}, Problems~\eqref{pb:LIMID} and \eqref{pb:NLP} are equivalent, and in particular admit the same optimal solutions in terms of $\delta$.

We recall that there always exists at least one optimal policy which is deterministic (and therefore integral) for Problem~\eqref{pb:LIMID}, that is a policy $\delta$ such that
\begin{equation}\label{eq:integralityDelta}
\delta_{v|\prt{v}}(x_{\fa{v}}) \in \{0,1\}, \quad \forall x_{\fa{v}} \in \calX_{\fa{v}},\: \forall v \in \Va.
\end{equation}
We can therefore add integrality constraint~\eqref{eq:integralityDelta} to~\eqref{pb:NLP}.
With this integrality constraint, Equation~\eqref{eq:NLP_fact_Va} becomes a \emph{logical constraint}, \ie a constraint of the form $\lambda y = z$ with $\lambda$ binary and continuous $y$ and $z$.
Such constraints can be handled by modern MILP solvers such as \texttt{CPLEX} or \texttt{Gurobi}, that can therefore directly solve Problem~\eqref{pb:NLP}.
Alternatively, by a classical result in integer programming, we can turn Problem~\eqref{pb:NLP} into an equivalent MILP by replacing constraint \eqref{eq:NLP_fact_Va} by its McCormick relaxation \citep{Mccormick:1976}.
For a given $\mathfrak{p}$, 
let $b$ be a vector of upper bounds $b_{\D{v}}(x_{\D{v}})$ satisfying

\begin{equation}
    \label{eq:admissible_bound}
    \bbP_{\delta'}\big( X_{\D{v}} = x_{\D{v}} \big) \leq b_{\D{v}}(x_{\D{v}})
\qquad \forall \delta'\in\Delta, \quad \forall v \in \Va, \quad \forall x_{\D{v}} \in \calX_{\D{v}}.
\end{equation}

\noindent For such a vector $b$,
we say that, for a given node $v$, $(\mu_{C_v},\delta_{v|\prt{v}})$ satisfies McCormick's inequalities (see \cref{app:McCormick}) if 
\begin{equation}\label{eq:McCormick}
\left\{ \begin{array}{l}
\displaystyle\mu_{C_v} \geq \mu_{\D{v}}+ (\delta_{v|\prt{v}} -1) \, b_{\D{v}},\\
\displaystyle\mu_{C_v} \leq \delta_{v|\prt{v}}  \, b_{\D{v}}, \\
\mu_{C_v} \leq \mu_{\D{v}}.
\end{array} \right.
\tag{${\rm McCormick}(v,b)$}
\end{equation}

Note that the last inequality $\mu_{C_v} \leq \mu_{\D{v}}$ can be omitted in our case as it is implied by the marginalization constraint $\mu_{\D{v}} = \sum_{x_v} \mu_{C_v}$ in the definition of $\calL_\calG$.
Given the upper bounds provided by $b$, we introduce the polytope of valid moments and decisions satisfying all McCormick constraints:
\begin{equation}\label{eq:PolytopePbDefinition}
    \calQ^b(G, \calX, \mathfrak{p}, \calG) = \Big\{(\mu,\delta) \in \calL_\calG \times \Delta \colon {\rm McCormick}(v,b) \text{ is satisfied for all } v \in \Va  \Big\}.
\end{equation}


With the previously introduced notation 
the MEU Problem~\eqref{pb:LIMID} is equivalent to the following MILP:
\begin{subequations}\label{pb:MILP}
\begin{alignat}{2}
\max_{\mu,\delta}  \enskip & \sum_{v \in \Vl} \langle r_v , \mu_v \rangle & \quad &\\
\text{\st} \enskip
& \mu \in \overline{\calP}(G, \calX, \mathfrak{p}, \calG) \label{eq:MILPcommon}\\
& \delta \in \detpol \\
& (\mu,\delta) \in \calQ^b
\end{alignat}
\end{subequations}
\noindent where $\detpol$ is the set of deterministic policies and contains the integrality constraints \eqref{eq:integralityDelta}.

\begin{rem}\label{rq:b_equal_1}
    The strength of the McCormick constraints~\eqref{eq:McCormick} depends on the quality of the bounds $b_{\D{v}}$ on $\mu_{\D{v}}$.
    As for a solution $\mu$ of Problem~\eqref{pb:MILP}, $\mu_{\D{v}}$ corresponds to a probability distribution,
    the simplest admissible bound over $\mu_{\D{v}}$ is simply $b = 1$.
    Unfortunately, McCormick's constraints are loose in this case: 
    we show in Appendix~\ref{sub:using_} that, for any $\mu$ in $\overline{\calP}$, there exists $\delta$ in $\Delta$ such that $(\mu,\delta)$ satisfies the McCormick constraints.
    Hence, when $b=1$, McCormick constraints fail to retain any information about the conditional independence statements encoded in the associated nonlinear constraints.
    Since $\delta$ does not appear outside of the McCormick constraints, their sole interest in that case is to enable the branching decisions on $\delta$ to have an impact on $\mu$.
%
%
Appendix~\ref{sub:mccormick_inequalities_with_well_chosen_bounds_are_useful} gives an example showing that McCormick constraints do retain information about the conditional independence if bounds $b_{\D{v}}$ smaller than $1$ are used.
Finally,
    Appendix~\ref{sub:algorithm_to_choose_good_quality_bounds} provides a dynamic programming algorithm that efficiently computes such a~$b$.
\end{rem}

\section{Valid cuts}
\label{sec:validCuts}

Classical techniques in integer programming such as branch and bound algorithms rely on solving the relaxation of the MILP to obtain a lower bound on the value of the objective.
For Problem~\eqref{pb:MILP} the relaxation is likely to be poor, and so the MILP is not well solved by off-the-shelf solvers: indeed as explained above, when $b=1$, the McCormick inequalities fail completely to enforce in the linear relaxation the conditional independences that are encoded in the nonlinear constraints, and using a better bound $b$ does not completely adress the issue.
In this section, we introduce valid cuts to strengthen the relaxation and ease the MILP resolution.
A \emph{valid cut} for a MILP is an (in)equality that is satisfied by any solution of the MILP, but not necessarily by solutions of its linear relaxation.
A family of valid cuts is stronger than another when the former yields a polytope strictly included in the latter.

\subsection{Constructing valid cuts}
\label{sub:valid_inequalities}



By restricting ourselves to vectors of moments $\mu \in \overline{\calP}$, we have imposed 
$$\bbP_{\mu}(X_v|X_{V\backslash\dsc{v}})=\bbP_{\mu}(X_v|X_{\D{v}})=p_{v|\prt{v}} \quad \text{for all $v$ in }\Vs, $$
because $\mu \in \overline{\calP}$ must satisfy $\mu_{C_v} = \mu_{\D{v}}p_{v|\prt{v}}$.
If we could impose as well the nonlinear constraints $\mu_{C_v} = \mu_{\D{v}}\delta_{v|\prt{v}}$ for $v$ in $\Va$, we would be able to impose that decisions encoded in $\mu$ at the nodes $a \in \Va$ satisfy $\bbP_{\mu}(X_a|X_{C_a \backslash \{a\}})=\bbP_{\mu}(X_a|X_{\prt{a}})$. Unfortunately, in general, 
The constraint $\mu_{C_v} = \mu_{\D{v}}p_{v|\prt{v}}$ for $v$ in $\Vs$ is linear only because $p_{v|\prt{v}}$ is a constant that does not depend on $\delta$.
But, as an indirect consequence of setting the conditional distributions $p_{v|\prt{v}}$ for $v \in \Vs$, there are other conditional distributions that do not depend on $\delta$.
Indeed, for some pairs of sets of vertices $C,D$ with $D\subseteq C$, the conditional probabilities $\bbP_{\delta}(X_D = x_D |X_{C\backslash D} = x_{C\backslash D})$ are identical for any policy $\delta$. We can therefore introduce valid cuts of the form
\begin{equation}\label{eq:validIneqalityForm}
	\mu_C = \mu_{C\backslash D} \, p_{D|C\backslash D}.
\end{equation}
While these additional constraints are not needed to set the value of the conditionals on $v \in \Vs$ and the conditional independences of the form $X_v \indep X_{V\backslash\dsc{v}} \mid X_{\prt{v}}$ for $v \in \Vs$, they can be useful to enforce some of the conditional independences that should be satisfied by $\mu$ at decision nodes. In particular, if  there exists a subset $M$ of $C\backslash D$ such that $p_{D|C\backslash D}=p_{D|M}$, then \eqref{eq:validIneqalityForm} enforces that for any $v \in \Va \cap (C\backslash (D \cup M)),$ we have $\bbP_{\mu}(X_a|X_{D\cup M})=\bbP_{\mu}(X_a|X_{M})$.
Clearly, the larger $D$, the stronger the valid cut. This motivates the following definition.


\begin{de}
Given a set of vertices $C$, we define $\B{C}$ to be the largest subset $D$ of $C$ such that, for any parametrization of $G$,
there exists $p_{D|C\backslash D}$ such that $\bbP_\delta(X_D|X_{C\backslash D}) = p_{D|C\backslash D}$ holds for any policy $\delta$.
We define $\M{C}$ as $C \backslash \B{C}$.
\end{de}

It is not obvious that a largest such set exists and is unique, and therefore that $\B{C}$ is well defined. We prove that it is the case later in this section. 
As for now, if we accept that $\B{C}$ is well defined, 
then the equalities
\begin{equation}\label{eq:validInequality}
	\mu_{C} = \mu_{\M{C}} p_{\B{C}|\M{C}}, \quad \forall C \in \calV,
\end{equation}
are the strongest valid cuts of the form \eqref{eq:validIneqalityForm} that we can obtain for Problem~\eqref{pb:MILP}.
We can then define $\Pfree$ as the polytope we obtain when we strengthen $\overline{\calP}$ with our valid cuts:
\begin{equation}
	\label{eq:Pfree}
	\Pfree(G,\calX,\mathfrak{p},\calG) =
	\left\{\mu \in \overline{\calP} \colon \mu_{C_v} = p_{\B{C}_v|\M{C}_v} \sum_{x_{\B{C}_v}} \mu_{C_v}  \text{ for all }v \in \Va\right\}.
\end{equation}

In the definition of $\Pfree$, we decided to introduce valid cuts of the form \eqref{eq:validInequality} only for sets of vertices $C$ of the form $C_v$ with $v \in \Va$. 
This is to strike a balance between the number of constraints added and the number of independences enforced.
Our choice is however heuristic, and it could notably be relevant to introduce constraints of the form \eqref{eq:validInequality} for well chosen $C \subsetneq C_v$.

Figure~\ref{fig:valid} provides an example of ID where valid cuts \eqref{eq:validInequality} reduce the size of the initial polytope. To compute $\B{C}$, we have used the characterization in the next section.

\begin{figure}
\begin{center}

\begin{tikzpicture}
\def\l{1.2}
\node[act] (a) at (0,0) {$a$};
\node[sta] (u) at (1,0) {$u$};
\node[sta] (v) at (2,0) {$v$};
\node[act] (b) at (3,0) {$b$};
\node[sta] (w) at (4,0) {$w$};

\draw[arc] (a) -- (u);
\draw[arc] (u) -- (v);
\draw[arc] (v) -- (b);
\draw[arc] (b) -- (w);
\draw[arc] (a) to[bend right] (w);
\draw[arc] (u) to[bend right] (w);
\end{tikzpicture}
\end{center}
\caption{Valid cut \eqref{eq:validInequality} with $C = \{a,u,v,b\}$ and $\B{C} = \{u\}$ is not implied by the linear inequalities of \eqref{pb:MILP}.
Indeed, suppose that $\calX_a = \calX_v = \{0\}$, while $\calX_u = \calX_b = \{0,1\}$.
Then the solution defined by $\mu_{auvb}(0,i,0,i)=0.5$ and $\mu_{auvb}(0,i,0,1-i)=0$ for $i\in \{0,1\}$ is in the linear relaxation of \eqref{pb:MILP} but does not satisfy \eqref{eq:validInequality}.
}
\label{fig:valid}
\end{figure}
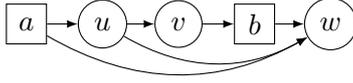
\subsection{Characterization of $\B{C}$}
\label{sub:characterization_of_Cfree}

In order to characterize $\B{C}$, we need some concepts from graphical model theory. The first notions make it possible to identify conditional independence from properties of the graph.
Let $D \subset V$ be a set of vertices. A trail ${v_1} \rightleftharpoons \dots \rightleftharpoons {v_n}$ is \emph{active given $D$} if,
whenever there is a v-structure ${v_{i-1}} \rightarrow {v_{i}} \leftarrow {v_{i+1}}$,
then $v_i$ or one of its descendant is in $D$, and no other vertex of the trail is in $D$.
Two sets of vertices $B_1$ and $B_2$ are said to be \emph{d-separated} by $D$ in $G$, and we will denote this property by $B_1\, \bot\, B_2 \mid D$, if there is no active trail between $B_1$ and $B_2$ given $D$.
We have $X_{B_1} \indep X_{B_2} \mid X_D$
for any distribution that factorizes on $G$
if and only if $B_1$ and $B_2$ are d-separated by $D$  \citep[Theorem 3.4]{koller2009probabilistic}.

The other notion we need is the one of \emph{augmented model} \citep[Chapter 21]{koller2009probabilistic}.
Consider $(G,\rho)$, a PID with $G=(\Vs,\Va,E)$, and let $V=\Va \cup \Vs$. 
For each $v \in \Va$, we introduce a vertex $\vartheta_v$ and a corresponding random variable $\theta_v$.
The variable $\theta_v$ takes its value in the space $\Delta_v$ of conditional distributions on $X_v$ given $X_{\prt{v}}$. 
Let $\Gaug$ be the digraph with vertex set $V_\Gaug = V \cup \vartheta_{\Va}$, where $\vartheta_{\Va} = \{\vartheta_v\}_{v \in \Va}$,
and arc set $\Arcs_\Gaug=\Arcs \cup \{(\vartheta_v,v), \forall v \in \Va \}$.
Such a graph $\Gaug$ is illustrated on \Cref{fig:example_valid}, where vertices in $\Gaug \backslash G$ are represented as rectangles with rounded corners.
The \emph{augmented model} of $(G,\rho)$ is the collection of distributions factorizing on $\Gaug$ such that $\calX_v$ is defined as in $\rho$ for each $v$ in $V$, $\calX_{\theta_v} = \Delta_v$, and 
\begin{equation}\label{eq:augmentedModel}
	\bbP\left(X_v = x_v | X_{\prt[\Gaug]{v}} = x_{\prt[\Gaug]{v}}\right) =
	\begin{cases}
	\theta^o_v(x_v | x_{\prt{v}}) & \text{if } v\in \Va, \\
	p_{v|\prt{v}}(x_v|x_{\prt{v}}) & \text{if } v\in \Vs,
	\end{cases}
\end{equation}
where $x_{\prt[\Gaug]{v}}=(x_{\prt[G]{v}},\theta^o_v)$ for $v \in \Va$, and $x_{\prt[\Gaug]{v}}=x_{\prt[G]{v}}$ for $v \in \Vs$.

A distribution of the augmented model is specified by choosing the distributions of the $\theta_v$.
In the rest of the paper, we denote by $\bbP_{\Gaug}$ the distribution of the augmented model with uniformly distributed $\theta_v$ for each $v$ in $\Va$.

With these definitions, a policy $\delta$ can now be interpreted as a value taken by $\theta_{\Va}$, and we have
\begin{equation}\label{eq:augmentedModelDsep}
	\bbP_{\delta}(X_D= x_D | X_M= x_M) = \bbP_{\Gaug}(X_D = x_D | X_M = x_M, \theta_{\Va} = \delta).
\end{equation}
Note that in general $\bbP_{\Gaug}(X_D = x_D | X_M = x_M)$ is the expected value over $\theta_{\Va}$ of $\bbP_{\theta_{\Va}}(X_D = x_D | X_M = x_M).$
The following result, which is an immediate consequence of \eqref{eq:augmentedModelDsep}, characterizes the pairs $(D,M)$ such that the conditional distribution $\bbP_{\delta}(X_D| X_M)$ is the same regardless of the choice of policy $\delta$.
\begin{prop}\label{prop:augmentedGraphIndependence}
We have $\bbP_{\delta}(X_D| X_M) = \bbP_{\Gaug}(X_D| X_M)$ for any PID on $G$, any policy $\delta$, and any $M$ such that $\bbP_{\delta}(X_M) > 0$ if and only if $D$ is d-separated from $\vartheta_{\Va}$ given $M$ in $\Gaug.$
\end{prop}
Note that this is a particular case of a result known in the causality theory for graphical models~\citep[see \eg][Proposition 21.3]{koller2009probabilistic}.
We have now all the tools to characterize $\B{C}$.

\begin{theo}\label{theo:CfreeCharacterization}
$\B{C}$ exists, is unique, and equal to
$\Big\{ v \in C \colon v \perp \vartheta_{\Va} \,| \, C \backslash \{v\} \Big\}$.
\end{theo}
\noindent With this characterization, the reader can check the value of $\B{C}$ on the example of Figure~\ref{fig:valid}. 

If we want to use the valid cuts in~\eqref{eq:Pfree} in practice, we must to compute $\B{C}$ and $p_{\B{C}|\M{C}}$ efficiently.
Theorem~\ref{theo:CfreeCharacterization} ensures that $\B{C}$ is easy to compute using any d-separation algorithm (and more efficient algorithms are presumably possible),
and Proposition~\ref{prop:augmentedGraphIndependence} ensures that, if we solve the inference problem on the \rjt for an arbitrary policy, \eg one where decisions are taken with uniform probability, we can deduce $p_{\B{C}|\M{C}}$ from the distribution $\mu_C$ obtained.

Theorem~\ref{theo:CfreeCharacterization} is an immediate corollary of the following Lemma, recently proved by two of the authors \citep[Theorem 1]{VCohenAparmentier}.


\begin{lem}\label{lem:mbDefWithoutMbName}
Let $B$ and $C$ be two sets of vertices. Then
$M^* := \Big\{v \in C \backslash B \colon v \,\nperp\,  B \,|\,  C \backslash (B \cup \{v\}) \Big\}$
is a subset $M$ of $C$ such that
\begin{equation}\label{eq:mbInCdef}
	B \perp C \backslash \left( B \cup M \right) \mid M.
\end{equation}
Furthermore, if $M$ satisfies~\eqref{eq:mbInCdef}, then $M^* \subseteq M, $ so that $M^*$ is a minimum for the inclusion.
\end{lem}
 \citet{VCohenAparmentier} call $M^*$ the \emph{Markov Blanket of $B$ in $C$.} Note that if $C=V$ this is the usual Markov Blanket.

\begin{proof}[Proof of Theorem~\ref{theo:CfreeCharacterization}]
Let $M$ be a subset of $C$. 
Proposition~\ref{prop:augmentedGraphIndependence} ensures that $\bbP_{\delta}(X_{C \backslash M}| X_{M})$ does not depend on $\delta$ for any parametrization if and only if $C \backslash M \perp \vartheta_{\Va} \mid M$. Theorem~\ref{theo:CfreeCharacterization} then follows by letting $B=\vartheta_{\Va}$ in Lemma~\ref{lem:mbDefWithoutMbName}.
\end{proof}

Using the terminology of \citet{VCohenAparmentier}, $\M{C}$ is the \emph{Markov blanket of $\vartheta_{\Va}$ in $C$}.

\begin{figure}
\begin{center}
\begin{tikzpicture}
\def\l{2.5}
\def\h{-2.5}

\draw[fill=blue, opacity=0.1] (-0.5*\l,0,-0.5*\h) -- (3.5*\l,0,-0.5*\h) -- (3.5*\l,0,3.6*\h) -- (-0.5*\l,0,3.6*\h);
\draw[thick] (-0.5*\l,0,-0.5*\h) -- (3.5*\l,0,-0.5*\h) -- (3.5*\l,0,3.6*\h) -- (-0.5*\l,0,3.6*\h) -- cycle;

\node at (-0.4*\l,-0.1*\l,2.3*\h) {$G$};

\node[sta] (s1) at (0*\l, 0,2*\h) {$s_1$};
\node[sta] (o1) at (0*\l, 0,0.5*\h) {$o_1$};
\node[uti] (r1) at (0*\l, 0,3*\h) {$r_1$};
\node[act] (a1) at (0.5*\l, 0,0.5*\h) {$a_1$};
\node[draw, rounded corners] (t1) at (0.5*\l, -2,0.5*\h) {$\vartheta_1$};
\node[sta] (s2) at (1*\l, 0,2*\h) {$s_2$};
\node[sta] (o2) at (1*\l, 0,0.5*\h) {$o_2$};
\node[uti] (r2) at (1*\l, 0,3*\h) {$r_2$};
\node[act] (a2) at (1.5*\l, 0,0.5*\h) {$a_2$};
\node[draw, rounded corners] (t2) at (1.5*\l, -2,0.5*\h) {$\vartheta_2$};
\node[sta] (s3) at (2*\l, 0,2*\h) {$s_3$};
\node[sta] (o3) at (2*\l, 0,0.5*\h) {$o_3$};
\node[uti] (r3) at (2*\l, 0,3*\h) {$r_3$};
\node[act] (a3) at (2.5*\l, 0,0.5*\h) {$a_3$};
\node[draw, rounded corners] (t3) at (2.5*\l, -2,0.5*\h) {$\vartheta_3$};
\node[sta] (s4) at (3*\l, 0,2*\h) {$s_4$};

\draw[arc] (t1) -- (a1);
\draw[arc] (t2) -- (a2);
\draw[arc] (t3) -- (a3);

\node (tar) at (3,0.6*\h) {};

\draw[arc] (s1) -- (s2);
\draw[arc] (s1) -- (o1);
\draw[arc] (o1) -- (a1);
\draw[arc] (a1) -- (s2);

\draw[arc] (s2) -- (o2);
\draw[arc] (o2) -- (a2);
\draw[arc] (s2) -- (s3);
\draw[arc] (a2) -- (s3);

\draw[arc] (s3) -- (o3);
\draw[arc] (o3) -- (a3);

\draw[arc] (a3) -- (s4);
\draw[arc] (s3) -- (s4);

\draw[arc] (s1) -- (r1);
\draw[arc] (s2) -- (r1);
\draw[arc] (a1) -- (r1);

\draw[arc] (s2) -- (r2);
\draw[arc] (s3) -- (r2);
\draw[arc] (a2) -- (r2);

\draw[arc] (s3) -- (r3);
\draw[arc] (s4) -- (r3);
\draw[arc] (a3) -- (r3);


\end{tikzpicture}
\end{center}
\caption{Example of augmented graph $\Gaug$ on a POMDP.}
\label{fig:example_valid}
\end{figure}
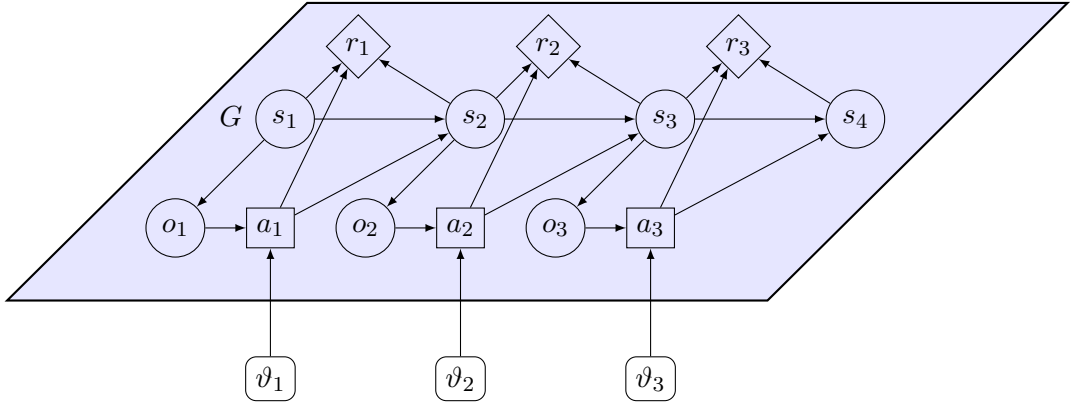









\subsection{Strength of the relaxations and their interpretation in terms of graph}
\label{sub:strength_and_interpretation_of_relaxations}
Consider $(G,\rho)$, a PID with $G = (\Vs,\Va,E)$ and $\rho = (\calX, \mathfrak{p},r)$.
Let $\calG$ be an \rjt on $G$, and $b$ an admissible bound satisfying~\eqref{eq:admissible_bound}.
The valid cuts of Section~\ref{sub:valid_inequalities} enable to introduce the following strengthened version of the MILP~\eqref{pb:MILP}.
\begin{equation}\label{eq:MILPstrengthened}
	\max_{\mu,\, \delta} \sum_{v \in \Vl} \langle r_v , \mu_v \rangle \text{ subject to}
	\enskip \mu \in \Pfree(G,\calX,\mathfrak{p},\calG),
	\enskip \delta \in \detpol,
	\enskip (\mu,\delta) \in \calQ^b.
\end{equation}
The following proposition summarizes the results of \Cref{sub:valid_inequalities}.
\begin{prop}
Any feasible solution $(\mu,\delta)$ of the MILP~\eqref{eq:MILPstrengthened}
is such that $\mu$ is the vector of moments of the distribution $\bbP_{\delta}$.
Hence, $(\mu,\delta)$ is an optimal solution of \eqref{eq:MILPstrengthened} if and only if $\delta$ is an optimal solution of the MEU problem~\eqref{pb:LIMID} on $(G,\rho)$. 
\end{prop}

In this section we give interpretations of the linear relaxations of~\eqref{pb:MILP} and~\eqref{eq:MILPstrengthened} in terms of graphs.
We introduce the sets of edges and IDs
\begin{subequations}
	\begin{alignat*}{2}
\overline{E} &= E \cup \big\{(u,v)\colon v \in \Va \text{ and } u \in C_v \backslash \fa{v}\big\}
&\quad \overline{G}&= (\Vs,\Va,\overline{E}),\\
\Efree &= E \cup \big\{(u,v) \colon v\in \Va \text{ and } u \in \M{C_v} \backslash \fa{v} \big\}
 &\quad\text{and}\quad\Gfree &= (\Vs,\Va,\Efree).
	\end{alignat*}
\end{subequations}
\noindent Figure~\ref{fig:boxeRelaxationsComparison} illustrates $\overline{G}$ and $\Gfree$ on the ID of Figure~\ref{fig:example_chess}.
Note that $E \subseteq \Efree \subseteq \overline{E}$, and remark the three following facts on $\overline{G}$ and $\Gfree$.
First, the definition of both IDs depends on $G$ and $\calG$.
Second, $\calG$ is still an \rjt on $\overline{G}$ and $\Gfree$.
And third, any parametrization $(\calX_V,\mathfrak{p},r)$ of $G$ is also a parametrization of $\overline{G}$ and of $\Gfree$.
The second and third results are satisfied by any ID $G' = (\Vs,\Va,E\cup E')$, where $E'$ contains only arcs of the form $(u,v)$ with $v \in \Va$ and $u \in C_v$.
Hence, if we denote by $\Delta_{G'}$ the set of feasible policies for $(G',\calX_V,\mathfrak{p},r)$, we can extend the definition of $\calS(G)$ in Equation~\eqref{eq:defQcalG} to such $G'$
$$ \calS(G') = \big\{
\mu \in \overline{\calP}  \colon \exists \delta \in \Delta_{G'}, \mu_{C_v} = \mu_{\D{v}} \delta_{v|\prt[G']{v}} \text{ for all $v$ in $\Va$}
\big\}.$$




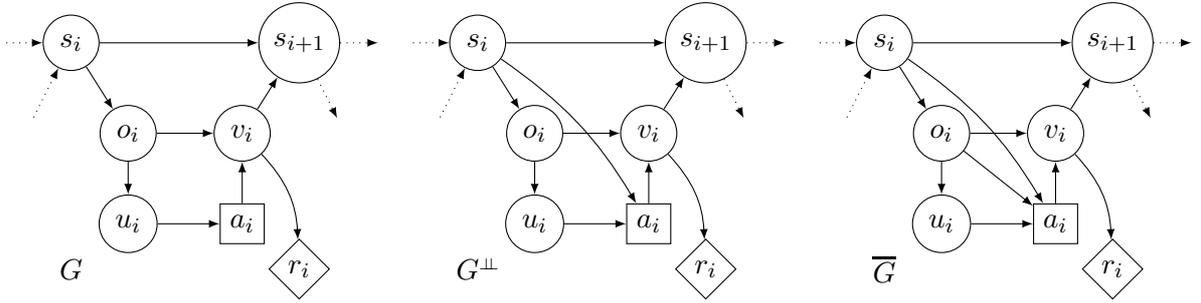
\begin{figure}
    \begin{center}
    \begin{tabular}{ccc}

        \begin{tikzpicture}
        \def\l{1.5}
        \def\h{1.2}

        \node[sta] (s1) at (-0.5*\l,2*\h) {$s_{i}$};
        \node[sta] (o1) at (0*\l,1*\h) {$o_{i}$};
        \node[sta] (u1) at (0*\l,0*\h) {$u_{i}$};
        \node[act] (a1) at (1*\l,0*\h) {$a_{i}$};
        \node[sta] (v1) at (1*\l,1*\h) {$v_{i}$};
        \node[uti] (r1) at (1.5*\l,-0.5*\h) {$r_{i}$};
        \node[sta] (s2) at (1.5*\l,2*\h) {$s_{i+1}$};

        \node at (-.5*\l,-0.5*\h) {$G$};

        \draw[arc] (s1) -- (s2);
        \draw[arc] (s1) -- (o1);
        \draw[arc] (o1) -- (v1);
        \draw[arc] (o1) -- (u1);
        \draw[arc] (u1) -- (a1);
        \draw[arc] (v1) -- (s2);
        \draw[arc] (a1) -- (v1);
        \draw[arc] (v1) to [bend left=20] (r1);

        \draw[arc,<-,dotted] (s1) -- ++ (-0.5,-1);
        \draw[arc,<-,dotted] (s1.west) -- ++ (-0.5,0);
        \draw[arc,dotted] (s2.east) -- ++ (0.5,0);
        \draw[arc,dotted] (s2) -- ++ (0.5,-1);

        \end{tikzpicture}

        &

                \begin{tikzpicture}
        \def\l{1.5}
        \def\h{1.2}

        \node[sta] (s1) at (-0.5*\l,2*\h) {$s_{i}$};
        \node[sta] (o1) at (0*\l,1*\h) {$o_{i}$};
        \node[sta] (u1) at (0*\l,0*\h) {$u_{i}$};
        \node[act] (a1) at (1*\l,0*\h) {$a_{i}$};
        \node[sta] (v1) at (1*\l,1*\h) {$v_{i}$};
        \node[uti] (r1) at (1.5*\l,-0.5*\h) {$r_{i}$};
        \node[sta] (s2) at (1.5*\l,2*\h) {$s_{i+1}$};

        \node at (-.5*\l,-0.5*\h) {$\Gfree$};

        \draw[arc] (s1) -- (s2);
        \draw[arc] (s1) -- (o1);
        \draw[arc] (o1) -- (v1);
        \draw[arc] (o1) -- (u1);
        \draw[arc] (u1) -- (a1);
        \draw[arc] (v1) -- (s2);
        \draw[arc] (a1) -- (v1);
        \draw[arc] (v1) to [bend left=20] (r1);

        \draw[arc,<-,dotted] (s1) -- ++ (-0.5,-1);
        \draw[arc,<-,dotted] (s1.west) -- ++ (-0.5,0);
        \draw[arc,dotted] (s2.east) -- ++ (0.5,0);
        \draw[arc,dotted] (s2) -- ++ (0.5,-1);

        \draw[arc] (s1) to[bend left=12] (a1);
        \end{tikzpicture}

        &

                \begin{tikzpicture}
        \def\l{1.5}
        \def\h{1.2}

        \node[sta] (s1) at (-0.5*\l,2*\h) {$s_{i}$};
        \node[sta] (o1) at (0*\l,1*\h) {$o_{i}$};
        \node[sta] (u1) at (0*\l,0*\h) {$u_{i}$};
        \node[act] (a1) at (1*\l,0*\h) {$a_{i}$};
        \node[sta] (v1) at (1*\l,1*\h) {$v_{i}$};
        \node[uti] (r1) at (1.5*\l,-0.5*\h) {$r_{i}$};
        \node[sta] (s2) at (1.5*\l,2*\h) {$s_{i+1}$};

        \node at (-.5*\l,-0.5*\h) {$\overline{G}$};

        \draw[arc] (s1) -- (s2);
        \draw[arc] (s1) -- (o1);
        \draw[arc] (o1) -- (v1);
        \draw[arc] (o1) -- (u1);
        \draw[arc] (u1) -- (a1);
        \draw[arc] (v1) -- (s2);
        \draw[arc] (a1) -- (v1);
        \draw[arc] (v1) to [bend left=20] (r1);

        \draw[arc,<-,dotted] (s1) -- ++ (-0.5,-1);
        \draw[arc,<-,dotted] (s1.west) -- ++ (-0.5,0);
        \draw[arc,dotted] (s2.east) -- ++ (0.5,0);
        \draw[arc,dotted] (s2) -- ++ (0.5,-1);

        \draw[arc] (o1) -- (a1);
        \draw[arc] (s1) to[bend left=12] (a1);
        \end{tikzpicture}
    \end{tabular}
    \end{center}
    \caption{Soluble relaxations corresponding to linear relaxations for the chess example.}
    \label{fig:boxeRelaxationsComparison}
\end{figure}


\begin{theo}\label{theo:PolytopesAndGraphs}
We have
$$\overline{\calP} = \calS(\overline{G}) \quad \text{and} \quad \max_{\mu \in \overline{\calP}}\displaystyle\sum_{v \in \Vl} \langle r_v , \mu_v \rangle = MEU(\overline{G},\rho),$$
and
$$\Pfree = \calS(\Gfree) \quad \text{and} \quad \max_{\mu \in \Pfree}\displaystyle\sum_{v \in \Vl} \langle r_v , \mu_v \rangle = MEU(\Gfree,\rho).$$
\end{theo}
Hence, if $(\mu,\delta)$ is a solution of the linear relaxation of~\eqref{pb:MILP}, then $\delta$ is a policy on $\overline{G}$, while if $(\mu,\delta)$ is a solution of the linear relaxation of~\eqref{eq:MILPstrengthened}, then $\delta$ is a policy on $\Gfree$.

Remark furthermore that $\calS(G')$ is generally not a polytope.
Indeed, when $G' = G$, this is the reason why $\eqref{pb:NLP}$ is not a linear program.
An important result of the theorem is that $\calS(\overline{G})$ and $\calS(\Gfree)$ are polytopes, and $MEU(\overline{G},\rho)$ and $MEU(\Gfree,\rho)$ can therefore be solved using the linear programs~$\max_{\mu \in \overline{\calP}}\displaystyle\sum_{v \in \Vl} \langle r_v,\mu_v \rangle$ and~$\max_{\mu \in \Pfree}\displaystyle\sum_{v \in \Vl} \langle r_v , \mu_v \rangle$ respectively.

The proof of the theorem uses the following lemma.

\begin{lem}\label{lem:independenceCfreeAndV}
Let $v$ be a vertex in $\Va$.
Then $x_{C_v}\mapsto p_{\B{C_v}|\M{C_v}}(x_{\B{C_v}}|x_{\M{C_v}\backslash v},x_v)$ is a function of $(x_{\B{C_v}},x_{\M{C_v}\backslash v})$ only. Hence, if a distribution $\mu_{C_v}$ satisfies $\mu_{C_v}=\mu_{\M{C_v}} p_{\B{C_v}|\M{C_v}}$, then $\B{C_v} \perp v \mid \M{C_v} \backslash\{v\}$ according to $\mu_{C_v}$.
\end{lem}

\begin{proof}
Consider the augmented model $\bbP_{\Gaug}$.
Let $P$ be a $\B{C_v}$-$v$ trail.
Let $Q$ be the trail $P$ followed by the arc $(v,\vartheta_v)$.
Given that $v$ has no descendants in $C_v$ (because of the hypothesis $\offspring{C}_v = \{v\}$), the vertex $v$ is a v-structure of $Q$.
As $v\in \M{C_v}$, if $P$ is active given $\M{C_v} \backslash \{v\}$, then $P$ is active given $\M{C_v}$, which contradicts the definition of $\B{C_v}$. Hence, $\B{C_v} \perp v \mid \M{C_v} \backslash \{v\}$ according to $\bbP_{\Gaug}$, and $x_{C_v}\mapsto p_{\B{C_v}|\M{C_v}}(x_{\B{C_v}}|x_{\M{C_v}\backslash v},x_v)$ is a function of $(x_{\B{C_v}},x_{\M{C_v}\backslash v})$ only.
The second part of the lemma is an immediate corollary.
\end{proof}

\begin{proof}[Proof of Theorem~\ref{theo:PolytopesAndGraphs}]
First, remark that,
once we have proved $\overline{\calP} = \calS(\overline{G})$ and $\Pfree = \calS(\Gfree)$,
the result follows from~Theorem~\ref{theo:NLPexact}.

We now prove $\overline{\calP} = \calS(\overline{G})$.
Let $\mu$ be in $\overline{\calP}$.
Then $\mu$ is a vector of moments in the local polytope of the \rjt $\calG$ on $\overline{G}$. Furthermore, since, first, for $v \in \Va,$ $\fa[\overline{G}]{v} = C_v$ , and second, for $v \in \Vs, \mu_{C_v} = \mu_{\D{v}} p_{v|\prt[G]{v}}$ together with  $\prt[\overline{G}]{v}=\prt[{G}]{v}$ imply that, according to $\mu_{C_v}$,  $X_v$ is independent from its non-descendants in $\overline{G}$ restricted to $C_v$ given $\prt[\overline{G}]{v}$ , Theorem~\ref{theo:rjtFactorization} ensures that $\mu$ is a vector of moments of a distribution that factorizes on $\overline{G}$, which yields $\overline{\calP} \subseteq \calS(\overline{G})$. Inclusion $\calS(\overline{G}) \subseteq \overline{\calP}$ is immediate.

Consider now a vector of moments $\mu$ in $\Pfree$.
Given $v \in \Va$, Lemma~\ref{lem:independenceCfreeAndV} and the definition of $\Gfree$ ensure that, according to $\mu_{C_v}$, variable $X_v$ is independent from its non-descendants in $\Gfree$ in $C_v$, \ie $\B{C_v}\backslash\{v\}$,
given its parents in $\Gfree$, \ie $\M{C_v} \backslash v$.
If $v \in \Vs$, constraints $\mu_{C_v} = \mu_{\D{v}} p_{v|\prt{v}}$ still implies that $X_v$ is independent from its non-descendants in $C_v$ given its parents according to $\mu_{C_v}$, because by definition of $\Gfree$, for $v \in \Vs$, we have $\prt[\Gfree]{v}=\prt[G]{v}$. Theorem~\ref{theo:rjtFactorization} again enables to conclude that $\Pfree \subseteq \calS(\Gfree)$. Inclusion $\calS(\Gfree) \subseteq \Pfree$ is immediate.
\end{proof}

\section{Soluble influence diagrams}
\label{sec:soluble_}

In this section, we make the assumption that IDs are such that any vertex $v\in V$ has a descendant in the set of utility vertices $\Vr$, \ie $\Vs \cup \Va = \casc{\Vl}$.
The following remark explains why we can make this assumption without loss of generality.

\begin{rem}
Consider a parametrized ID $(G, \rho)$ where $G = (\Vs,\Va,E)$ and $\Vs$ is the union of chance vertices $\Vc$ and utility vertices $\Vr$.
Let $(G',\rho')$ be the ID obtained by removing any vertex that is not in $\Vr$ and has no descendant in $\Vr$ and restrict $\rho$ accordingly. 
If a random vector $X_V$ factorizes as a directed graphical model on $(V,E)$ and $V' \subseteq V$ is such that $\casc{V'} = V'$, then $X_{V'}$ factorizes as a directed graphical model on the subgraph induced by $V'$ with the same conditional probabilities $p_{v|\prt{v}}$.
Hence, given a policy $\delta$ on $(G,\rho)$ and its restriction $\delta'$ to $(G', \rho')$, we have $\bbE_{\delta}\big(\sum_{v \in \Vr} r_v(X_v)\big) = \bbE_{\delta'}\big(\sum_{v \in \Vr} r_v(X_v)\big)$ where the first expectation is taken in
$(G,\rho)$ and the second in $(G', \rho')$, and the two IDs model the same \meu problem.
\end{rem}

The proofs of this section are quite technical and can be found in~\Cref{sub:proofs}.

\subsection{Linear program for soluble influence diagrams}
\label{sub:linear_program_for_soluble_limids}
Consider an ID $G = (\Vs,\Va,E)$ with $\Vs$ being the union of chance vertices $\Vc$ and utility vertices $\Vl$.
Given a policy $(\delta_u)_{u \in \Va}$ and a decision vertex $v$, we denote $\delta_{-v}$ the partial policy $(\delta_u)_{u \in \Va \backslash v}$.
A policy $(\delta_v)_{v \in \Va}$ is called a \emph{local optimum} if
$$ \delta_v \in \argmax_{\delta'_v \in \Delta_v} \bbE_{\delta'_v,\delta_{-v}}\Bgp{\sum_{u\in \Vl} r_u(X_u)} \quad \text{for each vertex $v$ in $\Va$.} $$
It is a \emph{global optimum} if it is an optimal solution of \eqref{pb:LIMID}.
Two concepts, \emph{s-reachability} and the \emph{relevance graph} have been introduced in the literature to characterize when a local minimum is also global~\citep[see \eg][Chapter 23.5]{koller2009probabilistic}.
A decision vertex $u$ is \emph{s-reachable} from a decision vertex $v$ if $\vartheta_u$ is not d-separated from $\dsc{v}$ given $\fa{v}$:
\begin{equation}\label{eq:def_soluble}
	\vartheta_u \,\nperp_{G^{\dagger}}\, \dsc{v} \mid \fa{v}.
\end{equation}
The usual definition is $\vartheta_u \,\nperp_{G^{\dagger}}\, \dsc{v} \cap \Vl \mid \fa{v}$, but these definitions coincide in our setting, since we have assumed that $\dsc{v} \cap \Vl \neq \emptyset$ for any $v \in \Va$.
Intuitively, the definition of this concept is motivated by the fact that the choice of policy $\delta_v$ given $(\delta_w)_{w \neq v}$ depends on $\delta_u$ only if $u$ is s-reachable from $v$. Note that, for example, if $u \in \dsc{v},$ then $u$ is s-reachable from $v$.
The \emph{relevance graph} of $G$ is the digraph $H$ with vertex set $\Va$, and whose arcs are the pairs $(v,u)$ of decision vertices such that $u$ is s-reachable from $v$.
Finally, the \emph{single policy update} algorithm (SPU) \citep{lauritzen2001representing} is the standard coordinate ascent heuristic for IDs. It iteratively improves a policy $\delta$ as follows: at each step, a vertex $v$ is picked, and $\delta_v$ is replaced by an element in $\displaystyle\argmax_{\delta'_v \in \Delta_v} \bbE_{\delta'_v,\delta_{-v}}\Big(\sum_{u\in \Vl} r_u(X_u)\Big)$.

The following proposition characterizes a subset of IDs, called \emph{soluble IDs}, which are easily solved, and provides several equivalent criteria to identify them.

\begin{prop}\cite[Theorem 23.5]{koller2009probabilistic}
Given an influence diagram $G$, the following statements are equivalent and define a \emph{soluble} influence diagram.
\begin{enumerate}
	\item For any parametrization $\rho$ of $G$, any local optimum is a global optimum.
	\item For any parametrization $\rho$ of $G$, SPU converges to a global optimum in a finite number of steps\footnote{In fact, if the graph is soluble, and if the decision nodes are ordered in reverse topological order for the relevance graph, then SPU converges after exactly one pass over the nodes.
}.
	\item The relevance graph is acyclic.
\end{enumerate}
\end{prop}

Given a parametrized influence diagram $G$ and an \rjt $\calG$, we introduced in Equation~\eqref{eq:defQcalG} the notation $\calS(G)$ for the subset of the local polytope $\calL_G$ corresponding to moments of policies.

The following theorem introduces a new characterization of soluble IDs in terms of convexity.

\begin{theo}\label{theo:validEqualSoluble}
If $G$ is not soluble then there exists a parametrization $\rho$ such that, for any junction tree $\calG$, the set of achievable moments $\calS(\calG)$ is not convex.

If $G$ is soluble, Algorithm~\ref{alg:buildSolubleRjt} returns an \rjt such that $\Pfree = \calS(\calG)$ for any parametrization~$\rho$.

\end{theo}

%
%
%
%
%
%

The property of being soluble characterizes ``easy'' IDs that can be solved by SPU.
Theorems~\ref{theo:NLPexact} and~\ref{theo:validEqualSoluble} imply that, if $G$ is soluble, our MILP formulation~\ref{eq:MILPstrengthened} reduces to the linear program
$$\max_{\mu \in \Pfree}\displaystyle\sum_{v \in \Vl} \langle r_v , \mu_v \rangle $$
and is therefore ``easy'' to solve.
Of course, this property of being ``easy'' refers only to the decision part of the ID.
If a soluble ID is such that, given a policy, the inference problem is not tractable, both SPU and our MILP formulation will not be tractable in practice.
Theorem~\ref{theo:validEqualSoluble} is a corollary of Theorem~\ref{theo:PolytopesAndGraphs} and the following lemma, and both results are proved in Section~\ref{sub:proofs}. 

\begin{lem}\label{lem:SolubleEqualRJT}
There exists an \rjt $\calG$ such that $\Gfree = G$ if and only if $G$ is soluble.
Such an \rjt can be computed using Algorithm~\ref{alg:buildSolubleRjt}.
\end{lem}
Note that based on a topological order of the relevant graph, Algorithm~\ref{alg:buildSolubleRjt} proceeds by computing a \emph{maximal perfect recall graph} that contains graph $G$ and that assigns the same parent sets to elements of $\Vs,$ then uses a topological order of this graph to order the nodes of $G$ for the computation of a rooted junction tree.

\begin{algorithm}
\caption{Build a ``good'' \rjt for a soluble graph $G$ }
\label{alg:buildSolubleRjt}
\begin{algorithmic}[1]
\STATE \textbf{Input:} An ID $G = (\Vs,\Va,E)$.
\STATE \textbf{Initialize:} $E' = \emptyset$. 
\STATE Compute the relevance graph $H$ of $G$
\STATE Compute an arbitrary topological order $\preceq_H$ on $\Va$ for the relevance graph $H$ \label{step:topoH}
\STATE $E' \leftarrow E \cup \{(u,v) \in \Va \times \Va\colon u \preceq_H v \}$ 
\STATE $G' \leftarrow (V,E')$
\STATE $E'' \leftarrow E \cup \{(u,v) \in \Vs \times \Va \colon u \notin \dsc[G']{v} \}$ \label{step:app} 
\STATE $G'' = (V,E'')$
\STATE Compute an arbitrary topological order $\preceq$ on $G''$ \label{step:topo}
\STATE \textbf{Return} the result of Algorithm~\ref{alg:buildRJT} for $(G, \preceq)$
\end{algorithmic}
\end{algorithm}



\subsection{Comparison of soluble and linear relaxations}
\label{sub:comparison_of_soluble_and_linear_relaxations}

MILP solvers are based on (much improved) branch-and-bound algorithms that use the linear relaxation to obtain bounds. Their ability to solve formulation~\eqref{eq:MILPstrengthened} therefore depends on the quality of the bound provided by the linear relaxation.
As SPU solves efficiently soluble IDs, we could imagine alternative branch-and-bounds schemes that use bounds computed using SPU on ``soluble graph relaxation'' of influence diagrams.
We now formalize this notion and compare the two approaches.

A \emph{soluble graph relaxation} 
 of an ID $G = (\Vs,\Va,E)$ is a soluble ID $G' = (\Vs,\Va,E')$ where $E'$ is the union of $E$ and a set of arcs with head in $\Va$.
 Remark that Theorem~\ref{theo:PolytopesAndGraphs} can be reinterpreted as the link between soluble graph relaxation and linear relaxations.
And since $\calS(\overline{G}) = \overline{\calP}$ and $\calS(\Gfree) = \Pfree$, by Theorem~\ref{theo:validEqualSoluble}, $\Gfree$ and $\overline{G}$ are soluble, and therefore soluble graph relaxations of $G$.

Since any feasible policy for the ID $G$ is a feasible policy for a soluble graph relaxation $G'$, for any parametrization $\rho$, the value of $\mathrm{MEU}(G',\rho)$,
which can be computed by SPU, provides a tractable bound on $\mathrm{MEU}(G,\rho)$.
Soluble relaxations can therefore be used in branch-and-bound schemes for IDs, as proposed in \citet{khaled2013solving}.
To compare the interest of such a scheme to our MILP approach we need to compare the quality of the soluble graph relaxation and linear relaxation bounds.
Let $G'$ be a soluble graph relaxation  of $G$, applying Algorithm~\ref{alg:buildSolubleRjt} on $G'$ provides an \rjt such that $\Efree \subseteq E'$
Indeed, by Lemma~\ref{lem:SolubleEqualRJT}, $v$ is d-separated from $C_v \backslash \fa[G']{v}$ given $\prt[G']{v}$ in $G'$, and therefore also in $G$, which implies $\Efree \subseteq E'$. 
Thus, by Theorem~\ref{theo:PolytopesAndGraphs},
the bound provided by the linear relaxation of the MILP~\eqref{eq:MILPstrengthened} is at least as good as the soluble graph relaxation bound, and sometimes strictly better thanks to constraints $(\mu,\delta) \in \calQ^b$.

\section{Numerical experiments}
\label{sec:numerical}

In this section, we provide numerical experiments showcasing the results of the paper.
In particular, on two examples of varying size, we study the impact of the valid inequalities.
On such examples, we solve the MILP formulation~\eqref{pb:MILP} with improved McCormick bounds relying on~\Cref{sub:algorithm_to_choose_good_quality_bounds},
and valid inequalities from~\Cref{sec:validCuts} obtained from the \rjt of \Cref{alg:buildRJT}.
More precisely we solve
 $\max \left\{ \sum_{v \in \Vl} \langle r_v , \mu_v \rangle \mid (\mu,\delta) \in \calQ, \delta \in \detpol \right\}$
\noindent where $\calQ$ is one of the four following polytopes :
${\overline{\calQ}}^{1} = \left( \overline{\calP} \times \Delta \right) \cap \calQ^{1}$ (no cuts),
$\overline{\calQ}^b = \left( \overline{\calP} \times \Delta \right) \cap \calQ^{b}$ (McCormick only),
$\calQ^{\smallindep,1} = \left( \Pfree \times \Delta \right) \cap \calQ^{1}$ (independence cuts only),
$\calQ^{\smallindep,b} = \left( \Pfree \times \Delta \right) \cap \calQ^{b}$ (McCormick and independence cuts).


The difficulty of an instance can be roughly measured by the number of feasible deterministic policies \citep{maua2012solving}, \ie $\big \vert \detpol \big \vert$.
We have $ \big \vert \detpol \big \vert = \prod_{v \in \Va} \big \vert \calX_v \big \vert^{\prod_{u \in \prt{v}}\big \vert \calX_u \big \vert}$. Therefore, the difficulty depends exponentially on $\big \vert \calX_v \big \vert$ for $v \in \fa{\Va}$. In our examples, we assume that $\omega_a = \big \vert \calX_v \big \vert$ for all $v \in \fa{\Va}$ and $\omega_s = \big \vert \calX_v \big \vert$  for all $v \in V \backslash \fa{\Va}$.
Each instance is generated by first choosing $\omega_a$ and $\omega_s$.
We then draw uniformly on $[0,1]$ the conditional probabilities $p_{v|\prt{v}}$ for all $v \in V \backslash \Va$ and on $[0,10]$ the rewards $r_v$ for all $v \in \Vr$.
We repeat the process $10$ times, and obtain therefore $10$ instances of the same size.

The results are reported in \Cref{tab:numerical_results}.
The first column specifies the size of the problem, the second the approximate number of admissible strategies.
The third column indicates the cuts used.
In the last four columns, we report the integrity gap (\ie the relative difference between the linear relaxation and the best integer solution), the final gap (relative difference between best integer solution and best lower bound), the improvement obtained over the solution given by SPU and the (shifted geometric mean of the) computation time for each instance.
All gaps are given in percentage. Computing times are given in seconds and correspond to the shifted geometric mean of the time over $10$ instances. All values are averaged over the $10$ instances.
In the last column, we write TL when the time limit is reached for the 10 instances of the same size.
Sometimes, the time limit is reached only for some of the $10$ instances, and we end up with a non-zero average final gap together with an average computing time that is smaller than the time limit.

All mixed-integer linear programs have been written in \texttt{Julia} \citep{bezanson2017julia} with \texttt{JuMP} \citep{DunningHuchetteLubin2017} interface and solved using \texttt{Gurobi} 7.5.2.
Experiments have been run on a server with 192Gb of RAM and 32 cores at 3.30GHz.
For each program, we use a warm start solution obtained by running the SPU algorithm of \citet{lauritzen2001representing} on the instances.

For notational simplicity, and since it is unambiguous, in the rest of this section we use the same notation to refer to a given node of the graph and to refer to the random variable associated with this node.


\subsection{Bob and Alice daily chess game}

We consider the chess game example represented in \Cref{fig:example_chess}.
The beginning of the \rjt built by \Cref{alg:buildRJT} for this example is represented in \Cref{fig:chess game_RJT}
Since $\vartheta_{a_{t-1}} \,\nperp_{G^{\dagger}}\, \dsc{a_t} \mid \fa{a_t}$ for all $t \in [T]$, the chess game example is not a soluble ID, thus cannot be solved to optimality by SPU.
\Cref{tab:random_instances_chess game_results} reports results on the generated instances.

\begin{figure}
\begin{center}
\begin{tikzpicture}

\def\l{2}
\def\h{1}

\node[draw, rounded corners] (C1) at (-2*\l,1*\h) {---$s_1$};
\node[draw, rounded corners] (C2) at (-1*\l,1*\h) {$s_1$---$o_1$};
\node[draw, rounded corners] (C3) at (0*\l,1*\h) {$s_1 o_1$---$u_1$};
\node[draw, rounded corners] (C4) at (1.3*\l,1*\h) {$s_1 o_1 u_1$---$a_1$};
\node[draw, rounded corners] (C7) at (2.5*\l,1*\h) {$s_1 o_1 a_1$---$v_1$};
\node[draw, rounded corners] (C5) at (4*\l,2*\h) {$v_1$---$r_1$};
\node[draw, rounded corners] (C8) at (4*\l,0*\h) {$s_1 v_1$---$s_2$};

\draw[arc] (C1) -- (C2);
\draw[arc] (C2) -- (C3);
\draw[arc] (C3) -- (C4);
\draw[arc] (C7) -- (C5);
\draw[arc] (C4) -- (C7);
\draw[arc] (C7) -- (C8);
\end{tikzpicture}
\end{center}
\caption{RJT for the chess game. The element to the right of --- is the offspring $\offspring{C}_v$.}
\label{fig:chess game_RJT}
\end{figure}
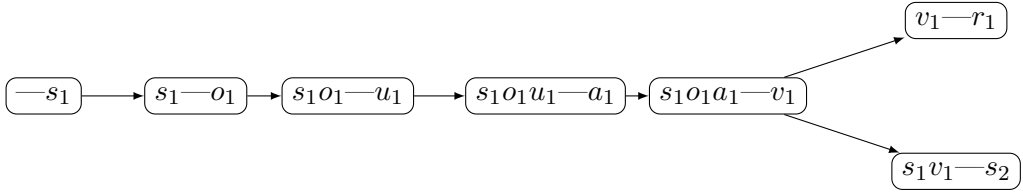


\begin{table}[!h]
\renewcommand\arraystretch{0.8}
{\scriptsize
\begin{subtable}{\textwidth}
\centering
    \begin{tabular}{|c|c|S|SScc|}
        \hline
        \multicolumn{1}{|c}{$(\omega_s,\omega_a,T)$} & \multicolumn{1}{c}{$|\Delta|$}&
        \multicolumn{1}{c}{Polytope} & \multicolumn{1}{c}{Int. Gap } & \multicolumn{1}{c}{Final Gap } & \multicolumn{1}{c}{SPU Gap } & \multicolumn{1}{c}{Time (s)}  \vline\\ \hline

        $(3,4,20)$& $10^{48}$ & ${\overline{\calQ}}^{1}$   & 5.18 & 0.35 & 0.02 &  299.1  \\
              &   & ${\overline{\calQ}}^{b}$ & 4.62 & 0.28 & 0.02 & 264.1 \\
              &   & $\calQ^{\smallindep,1}$  & 1.00 & 0.20 & 0.02 & 76.8 \\
              &   & $\calQ^{\smallindep,b}$  & 0.90 & 0.24 & 0.02 & 57.1 \\ \hline

        $(3,5,20)$ & $10^{69}$ & ${\overline{\calQ}}^{1}$ & 5.49 & 0.46 & 0.13 & 498.1  \\
            &    & ${\overline{\calQ}}^{b}$ & 5.05 & 0.46 & 0.13 & 562.3 \\
             &    & $\calQ^{\smallindep,1}$ & 1.38 & 0.27 & 0.13 & 200.5 \\
             &    & $\calQ^{\smallindep,b}$ & 1.27 & 0.23 & 0.013 & 181.7 \\ \hline

        $(3,6,20)$ & $10^{93}$ & ${\overline{\calQ}}^{1}$ & 4.17 & 0.48 & 0.05 & 1948.9  \\
             &   & ${\overline{\calQ}}^{b}$ & 3.84 & 0.36 & 0.05 & 1563.1 \\
             &   & $\calQ^{\smallindep,1}$ & 0.73 & 0.20 & 0.05 & 594.9 \\
             &    & $\calQ^{\smallindep,b}$ & 0.69 & 0.20 & 0.05 & 1109.5 \\ \hline
        $(3,9,20)$ & $10^{171}$ & ${\overline{\calQ}}^{1}$ & 6.57 & 1.50 & 0.16 & 2752 .2 \\
             &    & ${\overline{\calQ}}^{b}$ & 5.97 & 1.90 & 0.13 & 3067.8 \\
             &    & $\calQ^{\smallindep,1}$ & 1.05 & 0.37 & 0.16 & 843.6 \\
             &    & $\calQ^{\smallindep,b}$ & 1.00 & 0.37 & 0.13 & 868.2 \\ \hline
        $(3,10,20)$ & $10^{200}$ & ${\overline{\calQ}}^{1}$ & 6.99 & 2.28 & 0.04 & TL \\
             &    & ${\overline{\calQ}}^{b}$ & 6.45 & 2.39 & 0.04 & TL \\
             &    & $\calQ^{\smallindep,1}$ & 1.33 & 0.82 & 0.04 & 1759.3 \\
             &    & $\calQ^{\smallindep,b}$ & 1.25 & 0.81 & 0.04 & 1758.3 \\ \hline
        $(4,10,20)$ & $10^{200}$ & ${\overline{\calQ}}^{1}$ & 8.49 & 4.59 & 0.14 & TL \\
             &    & ${\overline{\calQ}}^{b}$ & 8.10 & 4.97 & 0.03 & TL \\
             &    & $\calQ^{\smallindep,1}$ & 2.40 & 1.77 & 0.11 & TL \\
             &    & $\calQ^{\smallindep,b}$ & 2.26 & 1.74 & 0.13 & TL \\ \hline
  \end{tabular}
  \caption{Results on chess game example
   \label{tab:random_instances_chess game_results}}
\end{subtable}

\vspace{0.5cm}

\renewcommand\arraystretch{0.8}
\begin{subtable}{\textwidth}
\centering
     \begin{tabular}{|c|c|S|SSSc|}
        \hline
        \multicolumn{1}{|c}{$(\omega_s,\omega_a,T)$} & \multicolumn{1}{c}{$|\Delta|$}&
        \multicolumn{1}{c}{Polytope} & \multicolumn{1}{c}{Int. Gap } & \multicolumn{1}{c}{Final Gap } & \multicolumn{1}{c}{SPU Gap } & \multicolumn{1}{c|}{Time (s)}  \vline\\ \hline

        $(3,4,20)$ & $10^{48}$ & ${\overline{\calQ}}^{1}$ & 5.34 & 0.58 & 4.15 & 563.1 \\
          &  & ${\overline{\calQ}}^{b}$ & 4.99 & 0.41 & 4.15 & 384.6 \\
          &  & $\calQ^{\smallindep,1}$ & 1.36 & Opt & 4.15 & 77.8 \\
          &  & $\calQ^{\smallindep,b}$ & 1.11 & Opt & 4.15 & 71.2 \\ \hline
        $(3, 5, 20)$ & $10^{69}$ & ${\overline{\calQ}}^{1}$ & 7.74 & 3.90 & 0.73 & 1090.9 \\
          &  & ${\overline{\calQ}}^{b}$ & 7.23 & 3.60 & 0.69 & 985.8 \\
          &  & $\calQ^{\smallindep,1}$ & 1.85 & 0.78 & 0.73 & 282.5 \\
          &  & $\calQ^{\smallindep,b}$ & 1.46 & 0.79 & 0.73 & 245.6 \\ \hline
        $(3, 6, 20)$ & $10^{93}$ & ${\overline{\calQ}}^{1}$ & 9.01 & 5.68 & 0.74 & TL  \\
          &   & ${\overline{\calQ}}^{b}$ & 8.47 & 5.42 & 0.74 & TL \\
          &   & $\calQ^{\smallindep,1}$ & 1.67 & 1.02 & 0.74 & 1935.0 \\
          &   & $\calQ^{\smallindep,b}$ & 1.37 & 1.00 & 0.74 & 1533.8 \\ \hline

        $(3,9,20)$ & $10^{171}$ & ${\overline{\calQ}}^{1}$ & 8.09 & 5.94 & 1.67 & TL \\
          &  & ${\overline{\calQ}}^{b}$ & 7.60 & 5.47 & 1.71 & TL \\
          &  & $\calQ^{\smallindep,1}$ & 2.45 & 1.86 & 1.59 & 2729.6 \\
          &  & $\calQ^{\smallindep,b}$ & 2.07 & 1.87 & 1.60 & 2894.9 \\ \hline
        $(3,10,20)$ & $10^{200}$ & ${\overline{\calQ}}^{1}$ & 12.40 & 10.0 & 1.24 & TL \\
          &  & ${\overline{\calQ}}^{b}$ & 11.76 & 9.95 & 1.23 & TL \\
          &  & $\calQ^{\smallindep,1}$ & 4.45 & 3.86 & 1.05 & TL \\
          &  & $\calQ^{\smallindep,b}$ & 3.87 & 3.77 & 1.11 & TL \\ \hline
          $(4,8,20)$ & $10^{144}$ & ${\overline{\calQ}}^{1}$ & 12.90 & 9.89 & 1.20 & TL \\
          &  & ${\overline{\calQ}}^{b}$ & 12.00 & 9.70 & 1.23 & TL \\
          &  & $\calQ^{\smallindep,1}$ & 3.14 & 2.27 & 1.20 & TL \\
          &  & $\calQ^{\smallindep,b}$ & 2.43 & 2.22 & 1.23 & TL \\ \hline

  \end{tabular}
  \caption{Results on POMDP example
  \label{tab:random_instances_POMDP_results}}
\end{subtable}

  }
 \caption{Mean results on $10$ randomly generated instances with a time limit TL=$3600$s.
 All gaps are relative and given in $\%$.
 We note TL when the time limit is reached, in which case the gap between the best integer solution and the lower bound is reported, otherwise we write Opt to specify that the solver reached optimality.
 \label{tab:numerical_results}}

 \end{table}

In this problem we see that we can tackle large problems : we can reach optimality in less than one hour for a strategy set of size $10^{144}$, and find a small provable gap on even bigger instances.
Moreover, we see that the independance cuts reduce the computation time by a factor 100, whereas the improved McCormick bounds yield less impactfull improvements.

However, on this problem the SPU heuristic yields good results that are marginally improved
by our MILP formulation. On this problem the main interest of our formulation is the bounds obtained.
In the next problem we show better improvement.

\subsection{Partially Observed Markov Decision Process with limited memory}

Another classical example of ID is the Partially Observed Markov Decision Process (POMDP) introduced in Section~\ref{sec:introduction}.
Figure~\ref{fig:POMDP_limited_memory} provides the graph representation of the POMDP with limited information. Since $\vartheta_{a_{t-1}} \,\nperp_{G^{\dagger}}\, \dsc{a_t} | \prt{a_t}$ for all $t \in [T]$, this ID is not soluble. Figure~\ref{fig:POMDP_RJT_limited_memory_vcuts} represents the RJT built by Algorithm~\ref{alg:buildRJT}.

However, for $v \in \Va$, $C_v^{\smallindep} = \emptyset$ and the linear relaxation of Problem \eqref{pb:MILP} does not enforce all the conditional independences that are entailed by the graph structure. Indeed, Theorem~\ref{theo:PolytopesAndGraphs} ensures that the linear relaxation of the MILP \eqref{pb:MILP} corresponds to solving Problem \eqref{pb:LIMID} on the graph $\overline{G}$.
For this example $\overline{G}$ corresponds to the MDP relaxation, in which the decision maker knows the state $s_t$ when he makes the decision $a_t$.
Therefore, the conditional independences $s_t \perp a_t | o_t$ is no more satisfied.
Although we cannot enforce these independences with linear constraints, we propose slightly weaker independences: in particular,
we propose an extended formulation corresponding to the bigger \rjt represented in \Cref{fig:POMDP_RJT_limited_memory_vcuts} to enforce for $s_t$ to be conditionally independent from $a_t$ given $(s_{t-1},a_{t-1},o_t)$ for $t>1$.
In such a \rjt, we have $C_{a_t}^{\smallindep} = \big\{ s_t \big\}$.
Therefore, we can derive valid equalities ~\eqref{eq:validInequality} in $\Pfree$.
We demonstrate the efficiency of such inequalities by solving the different formulation on a set of instances, summed up in \Cref{tab:random_instances_POMDP_results}.

This example is harder to solve to optimality as we only reach strategy set of size $10^{72}$.
Further, we can see that there are some instances where SPU seems to reach a local maximum that
is improved by our MILP formulation. Once again the valid cuts significantly reduce the root linear
relaxation gap and the solving time, even on large instances.
\begin{figure}
\begin{center}
\begin{tikzpicture}

\def\h{1.5}

\node[sta] (s1) at (0,2*\h) {$s_1$};
\node[sta] (o1) at (0,1*\h) {$o_1$};
\node[uti] (c1) at (0,3*\h) {$c_1$};
\node[act] (a1) at (1,1*\h) {$a_1$};

\node[sta] (s2) at (2,2*\h) {$s_2$};
\node[sta] (o2) at (2,1*\h) {$o_2$};
\node[uti] (c2) at (2,3*\h) {$c_2$};
\node[act] (a2) at (3,1*\h) {$a_2$};

\node[sta] (s3) at (4,2*\h) {$s_3$};
\node[sta] (o3) at (4,1*\h) {$o_3$};
\node[uti] (c3) at (4,3*\h) {$c_3$};
\node[act] (a3) at (5,1*\h) {$a_3$};

\node[sta] (s4) at (6,2*\h) {$s_4$};
\node[sta] (o4) at (6,1*\h) {$o_4$};

\node (tar) at (3,0.6*\h) {};

\draw[arc] (s1) -- (s2);
\draw[arc] (s1) -- (o1);
\draw[arc] (a1) -- (c1);
\draw[arc] (o1) -- (a1);
\draw[arc] (a1) -- (s2);
\draw[arc] (s1) -- (c1);

\draw[arc] (s2) -- (o2);
\draw[arc] (o2) -- (a2);
\draw[arc] (s2) -- (c2);
\draw[arc] (a2) -- (c2);
\draw[arc] (a2) -- (s3);
\draw[arc] (s2) -- (s3);

\draw[arc] (s3) -- (c3);
\draw[arc] (s3) -- (o3);
\draw[arc] (o3) -- (a3);
\draw[arc] (a3) -- (c3);
\draw[arc] (a3) -- (s4);

\draw[arc] (s3) -- (s4);
\draw[arc] (s4) -- (o4);

\end{tikzpicture}
\end{center}
\caption{Partially Observed Markov Decision Process with limited memory}
\label{fig:POMDP_limited_memory}
\end{figure}
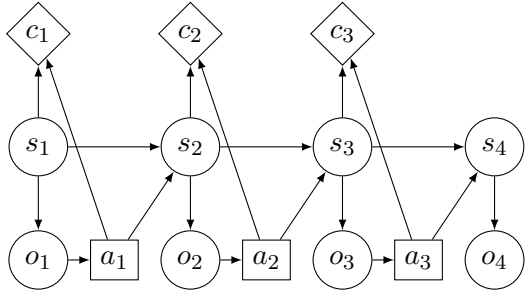

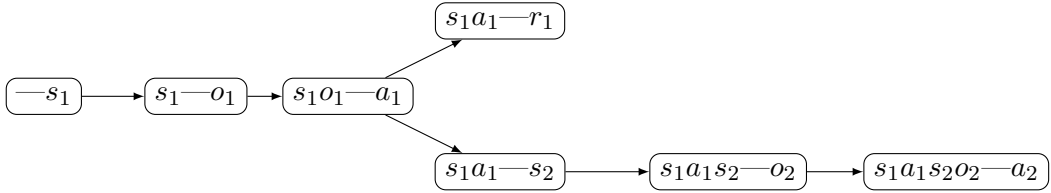
\begin{figure}
\begin{center}
\begin{tikzpicture}

\def\l{2}
\def\h{1}

\node[draw, rounded corners] (C1) at (-2*\l,1*\h) {---$s_1$};
\node[draw, rounded corners] (C2) at (-1*\l,1*\h) {$s_1$---$o_1$};
\node[draw, rounded corners] (C3) at (0*\l,1*\h) {$s_1 o_1$---$a_1$};
\node[draw, rounded corners] (C4) at (1*\l,0*\h) {$s_1 a_1$---$s_2$};
\node[draw, rounded corners] (C5) at (1*\l,2*\h) {$s_1 a_1$---$r_1$};
\node[draw, rounded corners] (C7) at (2.5*\l,0*\h) {$s_1 a_1 s_2$---$o_2$};
\node[draw, rounded corners] (C8) at (4*\l,0*\h) {$s_1 a_1 s_2 o_2$---$a_2$};
\draw[arc] (C1) -- (C2);
\draw[arc] (C2) -- (C3);
\draw[arc] (C3) -- (C5);
\draw[arc] (C3) -- (C4);
\draw[arc] (C4) -- (C7);
\draw[arc] (C7) -- (C8);
\end{tikzpicture}
\end{center}
\caption{Rooted Junction Tree for a Partially Observed Markov Decision Process with limited memory.}
\label{fig:POMDP_RJT_limited_memory_vcuts}
\end{figure}

\section*{Conclusion}
\label{sec:conclusion}

This paper introduced linear and mixed integer linear programming approaches for the \meu problem on influence diagrams.
The variables of the programs correspond, for the distributions induced by feasible policies, to the collection of vector of moments of the distribution on subsets of the variables that are associated to nodes of a new kind of junction tree, that we call a rooted junction tree.
We have thus introduced as well algorithms to build rooted junction trees tailored to our linear and integer programs.

For soluble IDs, which are IDs whose \meu problem is easy, in the sense that it can be solved by the single policy update algorithm, we showed that it can also be solved exactly via our linear programs.
Furthermore, we characterized soluble IDs as the IDs for which there exists a rooted junction tree such that the set of possible vector of moments on the nodes of the tree is convex for any parametrization of the influence diagram.

Finally, we proposed a mixed integer linear programming approach to solve the \meu problem on non-soluble IDs, together with valid cuts.
The bound provided by the linear relaxation is better than the bound that could be obtained using SPU on a soluble relaxation.
Numerical experiments show that the bound is indeed better in practice.






\bibliographystyle{plainnat}
\bibliography{stoOptimGraphModel}

\newpage 

\appendix
\section{Rooted junction tree properties}
\label{app:RJT}
In this section we present further technical results on \rjt that are usefull in the analysis of our approach.
We start we generic properties of \rjt.
\begin{prop}\label{prop:rjtProperties}
Let $\calG$ be an \rjt on $G$.
\begin{enumerate}
	\item\label{prop:rjtOrderCompatible}
	If there is a path from $u$ to $v$ in $G$, then there is a path from $C_u$ to $C_v$ in $\calG$.
	\item\label{prop:rjtProperties2} If $\cdsc[G]{u} \cap \cdsc[G]{v} \neq \emptyset$, then either there is a unique path from $C_u$ to $C_v$ or from $C_v$ to $C_u$ in $\calG$.
\end{enumerate}
\end{prop}

\begin{proof}
Let $\calG$ be an \rjt on $G$.
Consider a vertex $v$ of $G$ and a node $C$ of $\calG$ containing $v$.
Since $C$ is a node of $\calG_v$, and by definition of $C_v$, there exists a $C_v$-$C$ path in $\calG$.
Now consider $u\in \prt{v}$. Since $\fa{v} \subseteq C_v$, we have $u \in C_v$.
Thus there exists a $C_u$-$C_v$ path in $\calG$. The first statement follows by induction along a $u$-$v$ path in $G$.

We now show the second statement. Let $w$ be a vertex in $\cdsc[G]{u} \cap \cdsc[G]{v}$,
then by the first statement there exists both a $C_u$-$C_w$ and a $C_v$-$C_w$ path in $\calG$.
As $\calG$ is a rooted tree, this implies the existence of either a $C_u$-$C_v$ path or of a $C_v$-$C_u$ path in $\calG$.
\end{proof}

The following lemma is key in proving Theorem~\ref{theo:rjtFactorization}.
\begin{lem}\label{lem:rjtFactorization}
Let $C,D$ be subsets of $V$ such that $\fa{D} \subseteq C$ and $\cdsc{D} \cap C = D$. 
Any distribution $\mu_C$ on $C$ such that each $v$ in $D$ is independent from its non-descendants given its parents factorizes as
$\mu_C =\displaystyle \mu_{C\backslash D} \prod_{v \in D}q_{v|\prt{v}}$ where $\displaystyle\mu_{C\backslash D} = \sum_{x_D}\mu_C$ and $q_{v|\prt{v}}$ is defined as $\frac{\sum_{x_{C\backslash \fa{v}}}\mu_C}{\sum_{x_{C\backslash \prt{v}}}\mu_C}$ when the denominator is non-zero, and as $0$ otherwise.
\end{lem}
\begin{proof}
Let $\preceq$ be a topological order on $C$ such that $u \in C\backslash D$ and $v\in D$ implies $u\preceq v$.
Such a topological order exists since $\cdsc{D} \cap C = D$. We have
\begin{equation}
	\mu_C = \mu_{C \backslash D} \prod_{v \in D} \bbP_{\mu}(X_v | X_u, u \in C, u \prec v) = \mu_{C \backslash D} \prod_{v \in D} \bbP_{\mu}(X_v | X_{\prt{v}}),
\end{equation}
where the first equality is the chain rule and the second follows from the hypothesis of the lemma.
\end{proof}

\begin{proof}[Proof of Theorem~\ref{theo:rjtFactorization}]
Let $\calG$ be an \rjt on $G$.
Let $C_1,\ldots,C_n$ be a topological ordering on $\calG$, 
let $C_{\leq i} = \bigcup_{j \leq i}C_j$,
and $C_{<i} = C_{\leq_i} \backslash C_i$.
Let $\tau$ be a vector of moments satisfying the hypothesis of the theorem, and for each $v$ in $V$, let
$q_{v|\prt{v}}$ be equal to $\frac{\sum_{x_{C\backslash \fa{v}}}\tau_{C_v}}{\sum_{x_{C\backslash \prt{v}}}\tau_{C_v}}$ if the denominator is non-zero, and to $0$ otherwise.
We show by induction on $i$ that 
$$\mu_{C_{\leq i}} = \prod_{v \in C_{\leq i}} q_{v|\prt{v}} 
\quad\text{is such that}\quad
\tau_{C_{i'}} = \sum_{x_{C_{\leq i} \backslash C_{i'}}}\mu_{C_{\leq i}}
\quad \text{for all } {i'}\leq i.$$
Suppose the result true for all $j < i$, with the convention that $\mu_{0} = 1$.
We immediately deduce from the induction hypothesis that $\tau_{C_{i'}} = \sum_{x_{C_{\leq i} \backslash C_{i'}}}\mu_{C_{\leq i}}$ for all ${i'}< i $.
It only remains to prove to prove $\tau_{C_i} = \sum_{x_{C_{< i}}}\mu_{C_{\leq i}}$.
By definition of an \rjt, we have $\fa{\offspring{C_i}} \subseteq C_i$.
Proposition~\ref{prop:rjtProperties} implies that $\dsc{\offspring{C_i}} \cap C_i \subseteq \offspring{C_i}$.
Indeed let $u $ be in $\dsc{\offspring{C_i}} \cap C_i$. Then there is a $C_i$-$C_u$ path as $u \in \dsc{C_i}$, and a $C_u$-$C_i$ path as $u\in C_i$. Hence $C_u = C_i$ and $u \in \offspring{C_i}$.
By Lemma~\ref{lem:rjtFactorization}, we have $\tau_{C_i} = \tau_{C_i \backslash \offspring{C_i} } \prod_{v\in \offspring{C_i}} q_{v|\prt{v} }$. 
Let $C_j$ be the parent of $C_i$ in $\calG$, we have $\tau_{C_i \backslash \offspring{C_i}} = \sum_{x_{C_j\backslash C_i}} \tau_{C_j} 
= \sum_{x_{C_{< i} \backslash C_i}} \mu_{C_{<i}}$, the first equality coming from the fact that $(\tau_C)_{C \in \calV}$ belongs to the local polytope, and the second from the induction hypothesis.
Thus,
\begin{align*}
\sum_{x_{C_{<i}}}  \mu_{C_{\leq i}}
 = \sum_{x_{C_{<i}}} \prod_{v \in V_{\leq i}} q_{v|\prt{v}} 
 = \Big(\sum_{x_{C_{< i} \backslash C_i}} \mu_{C_{<i}}\Big)\prod_{v \in \offspring{C_i}} q_{v|\prt{v}}
 = \tau_{C_i \backslash \offspring{C_i}}\prod_{v \in \offspring{C_i}} q_{v|\prt{v}} = \tau_{C_i},
\end{align*}
which gives the induction hypothesis, and the theorem.

\end{proof}

\begin{proof}[Proof of Proposition~\ref{prop:algcorrect}]
Algorithm~\ref{alg:buildRJT} obviously converges given that it has only a finite number of iterations. If $G$ is not connected, the algorithm is equivalent to its separate application on each of the connected components, which each yield a tree. W.l.o.g., we prove properties of the algorithm under the assumption that $G$ is connected.
To simplify notations, we denote $C'_v$ by $C_v$, and check that it indeed corresponds to the root node of $v$.

We first prove that $\preceq$ is a topological order on $\calG$. First, remark that $(u \in C_v) \Rightarrow (u \preceq v)$. Indeed, if $u \in C_v,$ then either Step~\ref{step:Cvdef} of the algorithm ensures that $u \in \fa{v}$ and $u \preceq v$ or $u \notin \fa{v}$ and there exists $x$ such that  $u \in C_x$ and $C_v \rightarrow C_x$.
But by Step~\ref{step:arc} of Algorithm~\ref{alg:buildRJT}, the fact that $C_v \rightarrow C_x$ entails that $v$ is the maximal element of $C_x\backslash \{x\}$ for the topological order, so that $u \prec v$.
Furthermore, Step~\ref{step:arc} ensures that $(C_u,C_v) \in \calA$ implies $u \in C_v$.
We deduce from the previous result that $(C_u,C_v) \in \calA$ implies $u \preceq v$, and $\preceq$ is a topological order on $\calG$.

Then we show that \eqref{eq:minimality_rjt} holds. We first show  that $(u \in C_v) \Rightarrow C_u \rcarrow C_v$ and $u\in C'$ for any $C'$ on path $C_u\mcarrow C_v$. Either $u=v$ and this is obvious, or $u \in \prt[\calG]{C_v}$; and by recursion either $C_u \rcarrow C_v$ or $u \in C_r$ with $C_r$ the root of the tree which is also the first element in the topological order. But, unless $u=r$, this is excluded given that $u \in C_r$ implies $u \preceq r$. Note that this shows that $C_u$ is indeed the unique minimal element of the set $\{C \colon u \in C\}$ for the partial order defined by the arcs of the tree.
To show the first part of \eqref{eq:minimality_rjt}, we just need to note that either $u \in \fa{v}$ and the result holds, or there must exist $x$ such that $C_v \rightarrow C_x$ and $u \in C_x$ and by recursion, there exists $w$ such that $C_v \rcarrow C_w$ and $u \in \fa{w}$.

Finally, we prove that we have constructed an \rjt. Indeed, if two vertices $C_v$ and $C_{v'}$ contain $u$ then
since $\calG$ is singly connected, the trail connecting $C_v$ and $C_{v'}$ must be composed of vertices on the paths $C_v \lcarrow C_u$ and $C_u \rcarrow C_{v'}$, and we have shown in the previous paragraph that that $u$ belongs to any $C'$ on $C_v \lcarrow C_u$ and $C_u \rcarrow C_{v'}$, and so the running intersection property holds. Finally, property (ii) of Definition~\ref{def:rjt} must holds because the fact that $C_u$ is minimal among all cluster vertices containing $u$ together with the running intersection property entails that the cluster vertices containing $u$ are indeed a subtree of $\calG$ with root $C_u$.
\end{proof}

Proposition~\ref{prop:algcorrect} provides a justification for Algorithm~\ref{alg:buildRJT}, but it characterizes the content of the cluster vertices based on the topology of the obtained \rjt,
which is itself produced by the algorithm (note that the composition of cluster vertices depends only on $\preceq$ via the partial order of the tree). The cluster nodes of any \rjt and those produced by Algorithm~\ref{alg:buildRJT} admit however more technical characterizations using only $\preceq$ and the information in $G$, which we present next. These characterizations will be useful in \Cref{sub:proofs}.
For each vertex $v$ in $V$, let $$T_{\succeq v}=\{w \in V_{\succeq v} \mid \text{ there is a $v$-$w$ trail in } V_{\succeq v}\}.$$

\begin{prop}\label{prop:rjtInclusions}
Let $\calG = (\calV,\calA)$ be an \rjt satisfying $\offspring{C_v} = \{v\}$, and $\preceq$ be a topological order on $\calG$. Then $\preceq$ induces a topological order on $G$ and
\begin{subequations}\label{eq:rjtInclusions}
\begin{align}
	w\in T_{\succeq v}& \implies C_v \rcarrow C_w,
	\label{eq:dscCvContains} \\
	\left.\begin{array}{r}
	\cld{u} \cap T_{\succeq v} \neq \emptyset \\
	\text{ and }  u \preceq v
	\end{array}\right\}
	& \implies u \in C_v.
	\label{eq:CvContains}
\end{align}
\end{subequations}
\end{prop}


\begin{proof}
Let $\calG = (\calV,\calA)$ be an \rjt satisfying $\offspring{C_v} = \{v\}$, and $\preceq$ be a topological order on $\calG$.
Property~\ref{prop:rjtOrderCompatible} in Proposition~\ref{prop:rjtProperties} ensures that $\preceq$ induces a topological order on $G$.


We start by proving \eqref{eq:dscCvContains}. Let $v$ and $w$ be vertices such that $w \succ v$ and that there is a $v$-$w$ trail $Q$ in $V_{\succeq v}$. Let $s_0, \ldots, s_k$ be the nodes where $Q$ has a v-structure and $t_1,\ldots,t_k$ the nodes with diverging arcs in $Q$. Note that, since the trail is included in $V_{\succeq v}$, the first nodes of the trail have to be immediate descendants of $v$ in $G$ so that the trail takes the form $v \rcarrow s_0 \lcarrow t_1\rcarrow$ $s_1 \ldots t_k \rcarrow s_k \lcarrow w,$ where possibly $s_k=w$ and the last arc is not present. Then, given that $v\prec s_0$, and that $\preceq$ is topological for $\calG$, Proposition~\ref{prop:rjtProperties}.2 implies that ${C_v}\rcarrow {C_{s_0}}$.
But by the same argument, Property \ref{prop:rjtProperties2} in Proposition~\ref{prop:rjtProperties} implies ${C_{t_1}} \rcarrow {C_{s_0}}$, but since $\calG$ is a tree and $v \prec t_1$, we must have ${C_v}\rcarrow {C_{t_1}}\rcarrow {C_{s_1}}.$ By induction on $i$, we have ${C_v}\rcarrow {C_{s_i}}$ and thus ${C_v}\rcarrow {C_{w}},$ which shows \Cref{eq:dscCvContains}.


We now prove \eqref{eq:CvContains}. Let $u$ and $v$ be two vertices such that $u \preceq v$ and there is a $u$-$v$ trail $P$ with  $P\backslash \{u\}\subseteq V_{\succeq v}$. Let $w$ be the vertex right after $u$ on $P$. We have $u \in \fa{w},$ $w \succeq v$ and there is a $v$-$w$ trail in $V_{\succeq v}$, which implies $C_v \rcarrow C_w$ by \eqref{eq:dscCvContains}. But, since $u \preceq v$, the $u$-$v$ trail is also in $V_{\succeq u}$, which similarly shows that $C_u \rcarrow C_v$. So by \eqref{eq:rjt_prop} we have proved~\eqref{eq:CvContains}.
\end{proof}

\begin{prop}\label{prop:rjtBuildAlgorithm}
The graph $\calG = (\calV,\calA)$ produced by \Cref{alg:buildRJT} is the unique \rjt satisfying $\offspring{C_v} = \{v\}$ such that the topological order $\preceq$ on $G$ taken as input of \Cref{alg:buildRJT} induces a topological order on $\calG$ and the implications in \eqref{eq:rjtInclusions} are equivalences.
\end{prop}

\begin{proof}
Note that some visual elements of the proof are given in \Cref{fig:rjtAlgCharacterization}.
It is sufficient to prove the following inclusions
\begin{subequations}
\begin{align}
	\cdsc[\calG]{C_v} &\subseteq \{C_w \colon w \in T_{\succeq v}\}, \label{eq:revdscCvContains} \\
	C_v &\subseteq \{u\preceq v \colon \exists w \in T_{\succeq v}, \: u \in \fa{w}\}. \label{eq:revCvContains}
\end{align}
\end{subequations}
Indeed, note that by Proposition~\ref{prop:algcorrect}, the obtained tree is an \rjt so that, by Proposition~\ref{prop:rjtInclusions}, the reverse inequalities hold.

We prove the result by backward induction on \eqref{eq:revCvContains} and \eqref{eq:revdscCvContains}. For a leaf $C_v$ of $\calG,\: \cdsc[\calG]{C_v}=\{C_v\}$ so that \eqref{eq:revdscCvContains} holds trivially
and $C_v=\fa{v}$ so that \eqref{eq:revCvContains} holds because $u \in \fa{v}$ implies $u \preceq v$.
Then, assume the induction hypothesis holds for all children of a node $C_v$.

We first show \eqref{eq:revdscCvContains} for $C_v$ , i.e.\ that $(C_v \rcarrow C_w) \Rightarrow (w \in T_{\succeq v})$ (see Figure~\ref{fig:rjtAlgCharacterization}). Let $C_x$ be the child of $C_v$ on the path $C_v \rcarrow C_w$. By Proposition \ref{prop:algcorrect},
we have $v \prec x$, so that $V_{\succeq x} \subset V_{\succeq v} $.
Then, using the induction hypothesis, by \eqref{eq:revCvContains}, $(v \in C_x)$ implies that there is a $v$-$x$ trail in $V_{\succeq v}$, and by ~\eqref{eq:revdscCvContains}, $C_x \rcarrow C_w$ implies there is a trail $x$-$w$ in $V_{\succeq x}$, so there is a $v$-$w$ trail in  $v \prec x$ in $V_{\succeq v}$, which shows the result.

We then show \eqref{eq:revCvContains} for $C_v$ (see Figure~\ref{fig:rjtAlgCharacterization}). Indeed if $u \in C_v$, either $u \in \fa{v}$ and $u$ is in the RHS of \eqref{eq:revCvContains}, or there exists a child of $C_v$, say $C_x$ such that $u \in C_x$ and $u \prec v$, because the algorithm imposes $v=\max_{\preceq} (C_x \backslash\{x\})$. Since $C_v \rcarrow C_x$ there exists a path $v$-$x$ in $V_{\succeq v}$, and using induction, by \eqref{eq:revCvContains}, $(u \in C_x)$ implies that
$\exists w$ such that $u \in \fa{w}$ and there exists a trail $w$-$x$ in $T_{\succeq v}$. But we have shown in Proposition \ref{prop:algcorrect} that $(v \in C_x)\Rightarrow (v \preceq x)$, so $T_{\succeq x}\subset T_{\succeq v}$ and we have shown that there exists a $v$-$w$ trail in $T_{\succeq v}$ with $u \preceq v$ and $u \in \fa{w}$, which shows the result.
\end{proof}

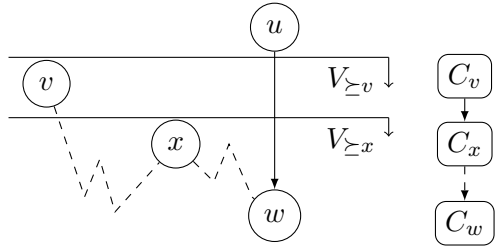
\begin{figure}
\begin{center}

\begin{tikzpicture}
\node[sta] (u) at (3,3) {$u$};
\node[sta] (w) at (3,0.5) {$w$};
\node[sta] (x) at (1.7,1.45) {$x$};
\node[sta] (v) at (0,2.25) {$v$};
\draw[arc] (u) -- (w);
\draw[dashed] (x) -- ++(0.5, -0.5) -- ++ (0.2,+0.5) -- ++ (0.3,-0.7) -- (w);
\draw[dashed] (v) -- ++(0.5,-1.5) -- ++ (0.2,0.5) --++ (0.2,-0.7)-- (x);
\draw (-0.5,2.60) -- (4.5,2.60);
\draw[->] (4.5,2.60) -- (4.5,2.20);
\draw (-0.5,1.8) -- (4.5,1.8);
\draw[->] (4.5,1.8) -- (4.5,1.55);
\node[below] at (4,2.60) {$V_{\succeq v}$};
\node[below] at (4,1.8) {$V_{\succeq x}$};
\node[draw,rounded corners] (Cv) at (5.5,2.35) {$C_v$};
\node[draw,rounded corners] (Cx) at (5.5,1.45) {$C_x$};
\node[draw,rounded corners] (Cw) at (5.5,0.4) {$C_w$};
\draw[ arc] (Cv) -- (Cx);
\draw[dashed, arc] (Cx) -- (Cw);
\end{tikzpicture}
\end{center}
\caption{Illustration of the proof of Proposition~\ref{prop:rjtBuildAlgorithm}. Plain arcs represent arcs, dashed line trails.}
\label{fig:rjtAlgCharacterization}
\end{figure}

 \section{Proofs of \Cref{sec:soluble_}}
 \label{sub:proofs}
For any set $C$ and binary relation $R$, we denote $C_{\mathbin{\mathrm{R}}v}$ the set of vertices $u$ in $C$ such that $u\mathbin{\mathrm{R}}v$.
The following lemma will be useful in the proof of Lemma~\ref{lem:SolubleEqualRJT}. Let $H$ denote the relevance graph of $G$.

\begin{lem}\label{lem:topoOrderSoluble}
 In general, $u \in \dsc{v} \Rightarrow (v,u) \in H$. But, when $G$ is soluble, then $u \in \dsc{v} \Leftrightarrow (v,u) \in H.$
\end{lem}
\begin{proof}
Assume that $u$ is s-reachable from $v$, that is $(v,u)$ is an arc in $H$. We first show that this implies that $u$ and $v$ have descendants in common. Indeed, by definition of  s-reachability, this means that there exist $w \in \dsc{v}$ and an active trail $T$ from $\vartheta_u$ to $w$. Either, $T$ is a directed path and $w$ is also a descendant of $u$ or $T$ must have a v-structure. In the latter case, let $x$ be the node with the v-structure closest to $\vartheta_u$ on $T$; since the trail is active, we must have $x \in \fa{v}$ but since $x$ is a descendant of $u$, in that case, $v$ must be a descendant of $u$. In both cases considered $u$ and $v$ have descendants in common.
Now, if $u$ is not a descendant of $v$, then $v$ is s-reachable from $u$, which is not possible as $H$ is acyclic.
Hence $u \in \dsc{v}$.
\end{proof}
As an immediate consequence, we obtain the following corollary.

\begin{coro}\label{coro:sharedTopo}
If $G$ is soluble and  $\preceq$ is a topological order on $G$, then its restriction $\preceq_H$ to $\Va$ is a topological order on the relevance graph $H$.
\end{coro}


\begin{proof}[Proof of Lemma~\ref{lem:SolubleEqualRJT}]
Let $G$ be a soluble influence diagram.
We start by proving that Algorithm~\ref{alg:buildSolubleRjt} with $G$ as an input returns an \rjt $\calG$.
It suffices to show that it is possible to compute topological orderings in
step~\ref{step:topo}, that is, to prove that $H$, $G'$ and $G''$, defined in \Cref{alg:buildSolubleRjt}, are acyclic.
$H$ is acyclic because the ID is soluble.
We now prove that $G'$ is acyclic. 
As $G$ is acyclic and by definition of $G'$, a cycle in $G'$ contains necessarily two vertices of $\Va$. Let $u$ and $v$ thus be two distinct elements of $\Va$.
Remark that, if there exists a path from $u$ to $v$ in $G$, then $v$ is strategically reachable from $u$, and $u \preceq_H v$.
Hence, by definition of $G'$, the indices of vertices in $\Va$ for $\preceq_H$ can only increase along a path in $G'$. There is therefore no cycle in $G'$ containing two vertices in $\Va$, and thus no cycle in $G'$.
We now prove that $G''$ is acyclic. 
Suppose that there is a cycle in $G''$.
Let $\preceq_{G'}$ be a topological order on $G'$, and let $v_h$
be the
smallest vertex $v$ for $\preceq_{G'}$ in the cycle such that there is an arc $(u,v)$ in $E'' \backslash E'$ in the cycle.
And let $(u_h,v_h)$ be the corresponding arc in the cycle.
Let $(u_l,v_l)$ be the arc of $E'$ right before $(u_h,v_h)$ in the cycle such that $v_l \in \Va$.
Arc $(u_l,v_l)$  is possibly identical to $(u_h,v_h)$.
By definition of $G'$, given two disjoint vertices $u$ and $v$ in $\Va$, either $(u,v) \in E'$ or $(v,u) \in E'$. 
Since $v_h \preceq_{G'} v_l$ by definition of $v_h$, we have either $v_h = v_l$ or $(v_h,v_l) \in E'$.
And since all the arcs in the $v_l$-$u_h$ subpath of the cycle are in $E'$, we have $u_h \in \cdsc[G']{v_l}$.
Hence $u_h \in \dsc[G']{v_h}$
, which contradicts the definition of $E''$ in Step~\ref{step:app}.
Hence, Algorithm~\ref{alg:buildSolubleRjt} always returns an \rjt, which we denote by~$\calG$.

It remains to prove that $\calG$ is such that $\M{C_v} \subseteq \fa{v}$ for each decision vertex $v\in \Va$.
We start with two preliminary results.
Remark that $E \subseteq E'$ implies that $\preceq$ is a topological order on $G$.
Let $\preceq_H$ denotes its restriction to $\Va$. 
Corollary~\ref{coro:sharedTopo}
ensures that $\preceq_H$ is a topological order on $H$.
Hence, we have
\begin{equation}\label{eq:extremalVertex}
	\vartheta_{\Va_{\prec v}} \perp \dsc{v} \mid \fa{v}, \quad
\text{for all }v \in \Va.\end{equation}
Furthermore, if
$u \in \Va$ and $v \in \Vs_{\succeq u} $, the definition of $G'$ implies the existence of a path from $u$ to $v$ in $G$, and hence the existence of a path from $\Va_{\succeq u}$ to $v$  in $G$.

We now prove $\M{C_v} \subseteq \fa{v}$ for each $v\in \Va$.
This part of the proof is illustrated on Figure~\ref{fig:proofOfTheoSoluble}.a.
Let $v$ be a vertex in $\Va$, let $u\in C_v \backslash \fa{v}$, and let $b \in \Va_{\prec u}$. 
We only have to prove that $u$ is d-separated from $\vartheta_b$ given $\fa{v}$.
We start by proving that $u$ and $v$ have common descendants.
Proposition~\ref{prop:rjtBuildAlgorithm} guarantees that \eqref{eq:CvContains} is an equivalence.
Hence, there exists a $u$-$v$ trail in $V_{\succ v}$.
Consider such a $u$-$v$ trail $Q$ with a minimum number of v-structures.
Suppose for a contradiction that $Q$ has more than one v-structure. 
Starting from $v$, let $w_1$ be the first v-structure of $Q$ and $u_1$ bet its first vertex with diverging arcs $u_1$.
Using the result at the end of the previous paragraph, we have $u_1 \in \cdsc{\Va_{\succeq v}}$.
Since $Q$ has been chosen with a minimal number of $v$-structures, we obtain $u_1 \in \cdsc{\Va_{\succ v}}$.
Let $a_1$ denote an ascendant of $u_1$ in $\Va_{\succ v}$.
Since $w_1 \in \dsc{v}$, Equation \eqref{eq:extremalVertex} ensures that $w_1 \perp \vartheta_v \mid \fa{a_1}$.
Hence, the $v$-$w_i$ path  is not active given $\fa{a_1}$, and it therefore necessarily intersects $\prt{a_i}$. Hence, $u_1 \in \dsc{v}$, and $Q$ there exists a $u$-$v$ trail $Q$ with fewer v-structures than $Q$, which gives a contradiction. Trail $Q$ therefore has a unique $v$-structure, and $u$ and $v$ have a common descendant $w$.
Hence, if $\vartheta_b$-$u$ trail $P$ is active given $\fa{v}$, then $P$ followed by a $u$-$w$ path is active given $\fa{v}$.
The fact that $\dsc{v} \perp \vartheta_b | \fa{v}$ ensures that there is no-such path $P$, and we have proved that $u$ is d-separated from $\vartheta_b$ given $\fa{v}$.

\begin{figure}
\def\l{1}
\def\h{0.7}

\begin{center}
a)
\begin{tikzpicture}
\draw (-0.5*\l,3.5*\h) -- (4*\l,3.5*\h);
\draw[->] (4*\l,3.5*\h) -- (4*\l,3.2*\h);
\node[below] at (3.7*\l,3.5*\h) {$V_{\succeq v}$};

\node[act](b ) at (3.5*\l,7*\h) {$b $};
\node[sta](u ) at (  0*\l,4.2*\h) {$u $};
\node[act](v ) at (  3*\l,  4*\h) {$v $};

\node[sta](x ) at (  0*\l,3.0*\h) {};
\node[sta](w) at (.5*\l,0.5*\h) {$w$};
\node[sta](u1) at (1.5*\l,1.5*\h) {$u_1$};
\node[act](a1) at (1.5*\l, 2.7*\h) {$a_i$};
\node[sta](w1) at (2.5*\l,.5*\h) {$w_1$};

\draw[dash dot] (v ) to node[very near start] (i) {} (w1);
\draw[dash dot] (w1) -- (u1) -- (w) to node[midway,left] {$Q$} (x );
\draw[arc, dash dot] (u ) -- (x );
\draw[dotted] (i) -- (a1) -- (u1);

\node[](pv) at (  2.5*\l,5.1*\h) {};
\node[](qv) at (  3.5*\l, 5.1*\h) {};
\node[draw,rounded corners, fit=(pv.east) (qv.south) (pv.west) (pv.north)] (prt) {};
\node[right=0.3 of prt.west] {$\prt{v}$};
\draw[arc] (prt.south) to (v );

\node (p1) at (1*\l,6.5 *\h) {};
\node (z1) at (1.5*\l,6 *\h) {};
\node (p2) at (2.5*\l,7*\h) {};
\node (z2) at (3*\l,6*\h) {};

\draw[dashed] (u ) to node[midway,left] {$P$} (p1.center) -- (z1.center) -- (p2.center) -- (z2.center) -- (b );
\draw[dotted] (z1) -- (prt);
\draw[dotted] (z2) -- (prt);
\end{tikzpicture}
\qquad
b)
\begin{tikzpicture}
\def\l{1}
\def\h{1}
\node[act] (v) at (3.0*\l,4.0*\h) {$v$};
\node[sta] (x) at (1.0*\l,3.5*\h+ 0.4*\h) {$x$};
\node[draw,rounded corners] (p) at (2.0*\l,2.5*\h) {$\prt{u}$};
\node[sta] (y) at (0.0*\l,2.0*\h) {$y$};
\node[act] (u) at (2.0*\l,1.5*\h) {$u$};
\node[sta] (z) at (0.0*\l,1.0*\h) {$z$};
\node[sta] (w) at (1.0*\l,0.0*\h) {$w$};

\node at (1.5*\l,3.0*\h+ 0.4*\h) (v1) {};
\node at (2.0*\l,3.5*\h+ 0.4*\h) (d) {};
\node at (2.5*\l,3.0*\h+ 0.4*\h) (v2) {};

\draw (v) -- (v2.center) -- (d.center) -- (v1.center) -- (x) -- (y) -- (z) -- (w);
\draw[dotted] (v1.center) -- (p) -- (v2.center);
\draw[dotted] (w) -- (u) -- (p);

\node[draw,rounded corners] (Cw) at (4.5*\l,0.0*\h) {$C_w$};
\node[draw,rounded corners] (Cz) at (4.5*\l,1.0*\h) {$C_z$};
\node[draw,rounded corners] (Cu) at (4.5*\l,2.0*\h) {$C_u$};
\node[draw,rounded corners] (Cy) at (4.5*\l,3.0*\h) {$C_y$};
\node[draw,rounded corners] (Cx) at (4.5*\l,4.0*\h) {$C_x$};

\draw[<-,dashed] (Cw) -- (Cz);
\draw[<-,dashed] (Cz) -- (Cu);
\draw[<-,dashed] (Cu) -- (Cy);
\draw[<-,dashed] (Cy) -- (Cx);

\end{tikzpicture}
\end{center}
\caption{Illustration o the proof of Lemma~\ref{lem:SolubleEqualRJT}. a) Direct statement, with $j = i-1$. Trail $P$ is illustrated by dashed line, trail $Q$ by a dash-dotted line, and other paths by dotted lines. b) Converse statement, with path $Q$ in plain line, and other paths dotted lines, and paths in $\calG$ in dashed lines.}
\label{fig:proofOfTheoSoluble}
\end{figure}
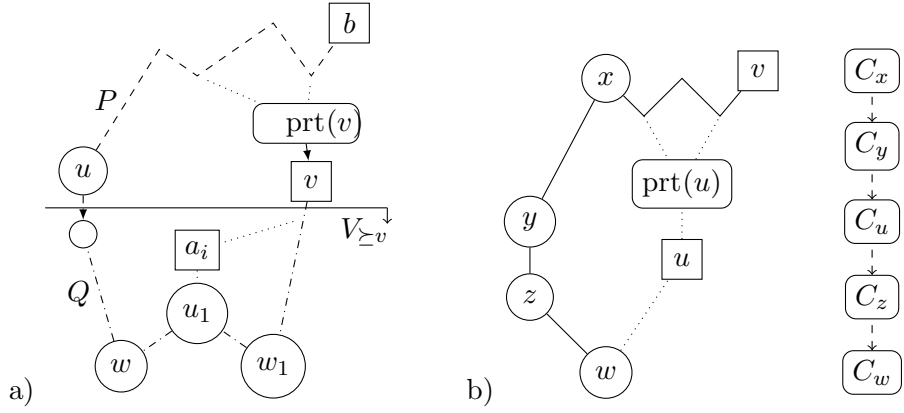

Conversely, let $G$ be a non-soluble influence diagram, and $\calG$ an RJT on $G$.
Let $u$ and $b$ be two vertices in $\Va$ such that $\dsc{v} \nperp \vartheta_{u} | \prt{v}$ and $\dsc{u} \nperp \vartheta_{v} | \prt{u}$.
Without loss of generality, we assume that if there is a path between $C_u$ and $C_v$, it is from $C_v$ to $C_u$. To prove the converse, we prove that $\M{C_u} \neq \fa{u}$.
This part of the proof is illustrated on Figure~\ref{fig:proofOfTheoSoluble}.b.
There exists an active trail $Q$ from $w \in \dsc{u}$ to $\vartheta_v$ given $\prt{u}$.
Starting from $w$, let $x$ be the first vertex with diverging arcs of $Q$ if $Q$ contains such a structure, and be equal to $v$ otherwise.
And let $P$ be the $w$-$x$ subtrail of $Q$.
Remark that $P$ must be an $x$-$w$ path in $G$, because any passing v-structure on $P$ cannot be at a descendant of $w$, for it would then be a descendant of $u$ which could not have any descendant in $\fa{u}$ as $G$ is acyclic.
The path $P$ contains no v-structure, and is active given $\fa{u}$. Hence, it does not intersect $\fa{u}$.
Since $x$ and $u$ have $w$ as common descendant, Proposition \ref{prop:rjtProperties} ensures that $C_x$ and $C_u$ are on the same branch of $\calG$.
If $v=x,$ $x \in \asc{w}$ and there is a path in $\calG$ from $C_x$ to $C_w$, moreover, since we assumed $C_u$ is a descendant of $C_v$ in $\calG$, and since $u \in \asc{w},$ then the path  from $C_x$ to $C_w$ contains $C_u$ and all the vertices of $P$.
Now, if $x \neq v$, then $x$ is the first vertex with diverging arcs, and in that case it belongs to $\asc{u}$, because $Q \setminus P$ must contain at least one v-structure and any such v-structure can only be at a node in $\asc{u}$.
So, again, there is a path in $\calG$ from $C_x$ to $C_w$ which contains $C_u$ and all the vertices of $P$.
Starting from $x$, let $y$ be the last vertex of $P$ such that $C_y$ is above $C_u$ in $\calG$, and $z$ be the child of $y$ in $P$.
But since $Q$ is active, the $y$-$\vartheta_{v}$ subtrail of $Q$ is active given $\fa{u}$, and we therefore have $\M{C_u} \neq \fa{u}$.
\end{proof}

\begin{proof}[Proof of Theorem~\ref{theo:validEqualSoluble}]
If $G$ is soluble, Lemma~\ref{lem:SolubleEqualRJT} ensures that Algorithm~\ref{alg:buildSolubleRjt} builds an \rjt $\calG$ such that $\Gfree = G$, and Theorem~\ref{theo:PolytopesAndGraphs} ensures that $\Pfree = \calS(P)$.





Consider now the result for non-soluble IDs. Let $G$ be a non-soluble influence diagram.
Let $a$ and $b$ be two decision vertices that are both strategically dependent on the other one.


\begin{figure}
\begin{center}
\begin{tikzpicture}
\def\l{-1}
\def\h{-1.1}
\def\stas{\phantom{$s_1$}}

\node[draw,rounded corners] (vta) at (7*\h,0*\l) {$\vartheta_a$};
\node[act] (a) at (6*\h,0*\l) {$a$};
\draw[arc,blue] (vta) -- (a);
\node at (6*\h,2*\l) {$s_{k-1}$};
\node[sta] (sk) at (6*\h,2*\l) {\stas};
\node[sta] (s1) at (6*\h,4*\l) {$s_1$};
\node[sta] (s0) at (6*\h,6*\l) {$s_0$};

\node[sta] (ap) at (5*\h,0.5*\l) {\stas};
\node[sta] (skp) at (5*\h,2.5*\l) {\stas};
\node[sta] (s1p) at (5*\h,4.5*\l) {\stas};
\node at (5*\h,0.5*\l) {$a'$};
\node at (5*\h,2.5*\l) {$s'_{k-1}$};
\node at (5*\h,4.5*\l) {$s'_1$};

\node[sta] (sks) at (5*\h,1.5*\l) {\stas};
\node[sta] (s1s) at (5*\h,3.5*\l) {\stas};
\node[sta] (s0s) at (5*\h,5.5*\l) {\stas};
\node at (5*\h,1.5*\l) {$s''_{k-1}$};
\node at (5*\h,3.5*\l) {$s''_1$};
\node at (5*\h,5.5*\l) {$s''_0$};

\node[sta] (tk) at (4*\h,1*\l) {$t_k$};
\node[sta] (t2) at (4*\h,3*\l) {$t_2$};
\node[sta] (t1) at (4*\h,5*\l) {$t_1$};
\node[sta] (pk) at (2.9*\h,1*\l) {$p_k$};
\node[sta] (p2) at (2.9*\h,3*\l) {$p_2$};
\node[sta] (p1) at (2.9*\h,5*\l) {$p_1$};
\node[act] (b) at (2*\h,3*\l) {$b$};
\node[sta] (wb) at (1*\h,3*\l) {$w_b$};
\node[sta] (wu) at (1*\h,5*\l) {$w_{s_0}$};
\node[sta] (w) at (0*\h,4*\l) {$w$};

\draw[arc,blue,dashed]  (a) -- (ap);
\draw[arc,blue,dashed] (sk) -- (skp);
\draw[arc,blue,dashed] (s1) -- (s1p);
\draw[arc,blue,dashed]  (ap) -- (tk);
\draw[arc,blue] (skp) -- (t2);
\draw[arc,blue] (s1p) -- (t1);

\draw[arc,blue,dashed] (sk) -- (sks);
\draw[arc,blue,dashed] (s1) -- (s1s);
\draw[arc,blue,dashed] (s0) -- (s0s);
\draw[arc,blue] (sks) -- (tk);
\draw[arc,blue] (s1s) -- (t2);
\draw[arc,blue] (s0s) -- (t1);

\draw[arc,dashed] (tk) -- (pk);
\draw[arc,dashed] (t2) -- (p2);
\draw[arc,dashed] (t1) -- (p1);
\draw[arc] (pk) -- (b);
\draw[arc] (p2) -- (b);
\draw[arc] (p1) -- (b);
\draw[arc,red] (b) -- (wb);
\draw[arc,red] (wb) -- (w);
\draw[arc,blue,dashed] (s0) to[bend right] (wu);
\draw[arc,blue] (wu) -- (w);

\node at (-0.5*\h,3*\l)[right, align=left]  {
	$ \calX_{s_0} = \calX_{s''_0} = \calX_{w_{s_0}} = \{1,2,\els\} $ \\
	$\calX_{s_\ell} = \calX_{s'_\ell} = \calX_{s''_\ell} = \{1,2\}$, $\forall \ell > 0$ \\
	$\calX_{t_\ell} = \calX_{p_\ell} = \{0,1\}$, $\forall \ell >0$ \\
	$\calX_{b} = \calX_{w_b} = \{1,2,\jok\}$ \\
	$\calX_{w} = \{-10,0,1,2\}$ \\
	\tikz[baseline=-0.5ex]{\draw[arc,red](0,0) -- (0.8,0);} P  \quad\quad \tikz[baseline=-0.5ex]{\draw[arc,blue](0,0) -- (0.8,0);} Q \\
	$p_{s_0}(x) = 1/3$ for $x$ in $\{1,2,\els\}$\\
	$p_{s_\ell}(x) = 1/2$ for $\ell>0$ and $x \in \{1,2\}$ \\
	$X_{t_\ell} = \ind(X_{s''_{k-1}} \neq X_{s'_\ell})$, $\forall \ell$ \\
	$X_{s''_\ell} = X_{s_\ell}, \quad X_{s'_\ell} = X_{s_i}$, $\forall \ell$ \\
	$X_{p_\ell} = X_{t_\ell}, \quad X_{w_b} = X_{b}$, $\forall \ell$ \\
	$X_{w_{s_0}} = X_{s_0}$  \\
	$X_w =
	\left\{\begin{array}{ll}
	0 & \text{if } X_{w_b} = \jok\\
	i & \text{if } X_{w_b} = X_{w_{s_0}} =i \text{ for } i \in \{1,2\}  \\
	-10 & \text{otherwise.}
	\end{array}\right.$

	};
\end{tikzpicture}
\end{center}
\caption{Influence diagram and parametrization used in the proof of Theorem~\ref{theo:validEqualSoluble}}
\label{fig:solubleEqualConvex}
\end{figure}
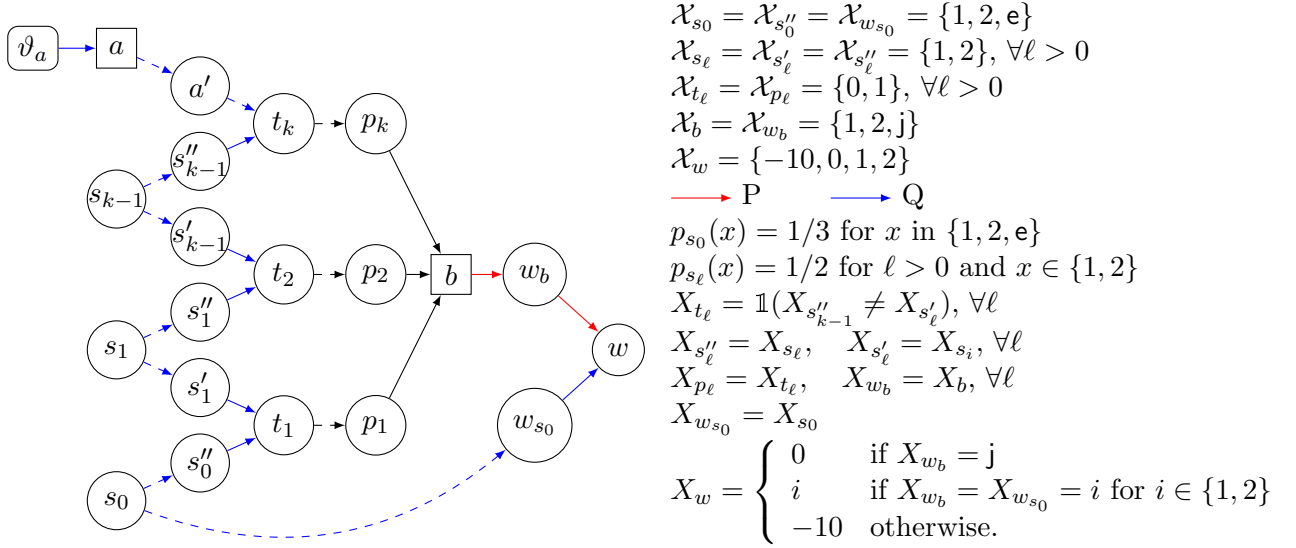


First, we suppose that $a \notin \asc{b}$ and $b \notin \asc{a}$. Let $P$ be a path from $a$ to $w \in \dsc{b}$ with a minimum number of arcs, and $Q$ be a $b$-$w$ path with a minimum number of arcs. Then $w$ is the unique vertex in the intersection of $P$ and $Q$.
Let $u$ and $v$ be the parents of $w$ in $P$ and $Q$ respectively.
Consider a parametrization where all the variables that are not in $P$ or $Q$ are unary, all the variables in $P$ and $Q$ are binary, all the variables in the $a$-$u$ subpath of $P$ are equal to $X_a$, all the variables in the $b$-$v$ subpath of $P$ are equal to $X_b$, and $p_{w|\prt{w}}$ is defined arbitrarily.
Let $\calG$ be an arbitrary junction tree, $C$ be its cluster containing $\fa{w}$.
Then choosing a distribution $\mu_a$ as policy $\delta_a$ and a distribution $\mu_b$ as policy $\delta_b$ implies that the restriction of $\mu_C$ to $X_{uv}$ is $\mu_{uv} = \mu_a \mu_b$. Hence, the marginalization on $X_{uv}$ of the set of distributions $\mu_C$ that can be reached for different policy is the set of independent distributions, which is not convex. Hence, $\calS(\calG)$ is not convex.

We now consider the case where $a \in \asc{b}$ or $b \in \asc{a}$.
W.l.o.g.,~we suppose $a \in \asc{b}$.
There exists a trail from $\vartheta_a$ to $w$ in $\dsc{b}$ that is active given $\prt{b}$. Let $Q$ be such a trail with a minimum number of $v$-structures.
And let $P$ be a $b$-$w$ path.
W.l.o.g., we suppose that $w$ is the only vertex in both $P$ and $Q$.
Let $w_b$ be the parent of $w$ on $P$ and $w_{s_0}$ its parent in $Q$.
Starting from $w$, let $s_0, \ldots,s_{k-1}$ denote the vertices with divergent arcs in $Q$, let $t_1,\ldots,t_k$ the $v$-structures, and $p_{\ell}$ denote the parent of $b$ that is below $t_{\ell}$.
Finally let $s'_{\ell}$ (resp.~$s''_{\ell}$) denote the parent of $t_{\ell}$ (resp~$t_{\ell+1}$) on the $s_{\ell}$-$t_{\ell}$ subpath (resp.~$s_{\ell}$-$t_{i+1}$ subpath) of $Q$.
The structures that we have just exhibited entail that $G$ contains a subgraph of the form represented on Figure~\ref{fig:solubleEqualConvex}.
Each dashed arrows correspond to a path whose length may be equal to $0$, in which case the vertices connected by the path are the same.


We now introduce a game that we will be able to encode on the graph of Figure~\ref{fig:solubleEqualConvex} and hence on $G$.
This game is a dice game with two players $a$ and $b$.
Before rolling a uniform die with three faces, player $a$ chooses to play $1$ or $2$, where ``playing $i$" means observing if the die is equal to $i$, and passing this information to player $b$.
The die $s_0$ is rolled.
If $a$ has played $1$ (resp.~$2$), he passes the information \textsf{true} to $b$ if the die $s_0$ is equal to $1$ (resp.~$2$), and $\textsf{false}$ if it is equal to $2$ (resp.~$1$), or something else $\els$.
Player $b$ does not know what $a$ has played.
Based on the information he receives, player $b$ decides to play $1$, $2$, or~joker, that we denote $\jok$.
If he plays $\jok$, then none of the player either earns or loses money. If he plays $i$ in $\{1,2\}$, then both players earn $i$ euros if die $s_0$ is equal to $i$, and lose 10 euros otherwise. The goal is maximize the expected payoff.
This game has two locally optimal strategies $\delta^1$ and $\delta^2$. In strategy $\delta^i$, player $a$ plays $i$ and $b$ plays $i$ if he receives \textsf{true} and $\jok$ otherwise.
Both strategy are locally optimal: each players decision is the best possible given the other ones. But only strategy $\delta^2$ is globally optimal.

It changes nothing to the game if we add $k-1$ coin tosses $X_{s_1},\ldots,X_{s_{k-1}}$, and player $b$ observes the $k$ equality tests $X_{t_1},\ldots, X_{t_{k-1}}$, where $X_{t_\ell} =  \ind(X_{s_{\ell-1}} = X_{s_\ell})$.
Indeed, player $b$ can compute $ \sum_{\ell=1}^k x_{p_\ell} $ and knows that $X_a = X_{s_0}$ if and only if this sum is even.
The parameterization of the influence diagram that enables to encode this game is specified on the right part of Figure~\ref{fig:solubleEqualConvex}.
For any $x$, the mapping $\ind_x(\cdot)$ is the indicator function of $x$.
All the variables that are not on Figure~\ref{fig:solubleEqualConvex} or on the paths on Figure~\ref{fig:solubleEqualConvex} are unary.
All the variables along paths represented by dashed arrows are equal.
Policies $\delta^i$ can therefore be defined as
\begin{equation*}
\delta_a^i = \ind_i \quad \text{and} \quad \delta_b^i(x_{p_1},\ldots,x_{p_k}) = \left\{
\begin{array}{ll}
i & \text{if } \sum_{\ell=1}^k x_{p_\ell} = 0 \bmod 2, \\
0 & \text{otherwise.}
\end{array}
\right.
\end{equation*}
where $\ind_i$ is the Dirac in $i$. 
A technical case to handle is the one where $a=a'=t_k$. In that case, we define $\calX_a = \{0,1\}$ and $\delta_a^i = \ind_i(X_{s_{k-1}}'') $.

Consider now a junction tree $\calG$ on $G$.
Let $C$ be a node of $\calG$ that contains $\fa{w}$. Then $C$ contains both $w_b$ and $w_{s_0}$.
Let $\mu_{C}^1$ and $\mu_{C}^ 2$ be the distributions induced by $\delta^1$ and $\delta^2$ on $X_{C}$, and $\mu_{bs_0}^1$ and $\mu_{bs_0}^2$ their marginalizations on $X_{w_bw_{s_0}}$.
Since $X_{w_b} = X_b$ and $X_{w_{s_0}} = X_{s_0}$, 
$\mu_{bs_0}^1$ and $\mu_{bs_0}^2$ are the distributions induced by $\delta^1$ and $\delta^2$ on $X_{bs_0}$.
Let $\mu_{bs_0} = \frac{\mu_{bs_0}^1 + \mu_{bs_0}^2}{2}$.
Denoting again $\ind_{x}$ the Dirac at $x$, we have 
\begin{equation*}
\mu_{bs_0}^1 =  \frac{\ind_{11} + \ind_{\jok 2} + \ind_{\jok \els}}{3}
,\quad
\mu_{bs_0}^2 =  \frac{\ind_{\jok 1} + \ind_{2 2} + \ind_{\jok \els}}{3},
\enskip \text{and} \enskip
\mu_{bs_0} = \frac{\ind_{11} + \ind_{\jok 1} + \ind_{\jok 2} + \ind_{2 2} + 2 \ind_{\jok \els}}{6}.
\end{equation*}
We claim that there is no policy $\delta$ that induces distribution $\mu_{bs_0}$ on $X_{bs_0}$.
Indeed, in a distribution induced by a policy $\delta$, it follows from the parametrization that
if $\bbP(X_a = 1) < 1$ and $\bbP(X_b = 1) > 0$, then $\bbP(X_b = 1, X_u = \els) > 0$.
And, if $\bbP(X_a = 2) < 1$ and $\bbP(X_b = 2) > 0$, then $\bbP(X_b = 2, X_u = \els) > 0$.
(In both claims, ``if $\bbP(X_a = 1) < 1$'' must be replaced by ``if $\delta_a(x_{s_{k-1}}) \neq \ind_i(x_{s_{k-1}})$'' when $a = t_k$).
As $\mu_{ub}$ is such that $\bbP(X_b = 1) > 0$, $\bbP(X_b = 2) > 0$, and $\bbP(X_b = 1, X_u = \els) = 0$, it cannot be induced by a policy.
Hence, $\calS(\calG)$ is not convex.
\end{proof}

\section{Algorithm to build a small \rjt}
\label{sec:algorithm_rjt}

In this section we present an algorithm to build a \rjt without considering a topological ordering on the initial graph $G = (V,A)$. 

\begin{algorithm}[H]
\caption{Build a \rjt}
\label{alg:build_smallest_RJT2}
\begin{algorithmic}[1]
\STATE \textbf{Input} $G = (V,E)$
\STATE \textbf{Initialize} $\mathcal{C} = \emptyset$ and $\calA' = \emptyset$
\STATE $L = \{v \ \text{s.t} \ \cld[G]{v} = \emptyset \}$
\WHILE{$L \neq \emptyset$}
\IF{$\Va \cap L \neq \emptyset$} \label{step:priority}
\STATE Pick $v \in \Va \cap L$ \label{step:priority2}
\ELSE
\STATE Pick $v \in L$
\ENDIF
\STATE $C_v \leftarrow \fa{v}$
\FOR{$C_x \in \mathcal{C}:  v \in C_x$}
\STATE $C_v \leftarrow C_v \cup ( C_x\backslash \{x\} )$
\STATE Remove $C_x$ from $\mathcal{C}$
\STATE Add $\{v,x\}$ in $\calA'$
\ENDFOR
\STATE Add $C_v$ to $\calC$
\STATE Remove $v$ from $G$ and recompute $L$
\ENDWHILE
\STATE $\calA \leftarrow \{(C_u,C_v) \mid (u,v) \in \calA'\}$
\STATE \textbf{Return} $\calG = ((C_v)_{v \in V} , \calA)$
\end{algorithmic}
\end{algorithm}

The only difference between Algorithms~\ref{alg:buildRJT}
and~\ref{alg:build_smallest_RJT2} 
is that the for loop along a reverse topological ordering of Algorithm~\ref{alg:buildRJT} is replaced in Algorithm~\ref{alg:build_smallest_RJT2} by a breadth first search that computes online this reverse topological ordering.
Hence, if we denote $\preceq$ this ordering, 
Algorithm~\ref{alg:build_smallest_RJT2} builds the same \rjt as the one we obtain when we use Algorithm~\ref{alg:buildRJT} with $\preceq$ in input.
Therefore, the \rjt built by Algorithm~\ref{alg:build_smallest_RJT2} satisfies~\ref{eq:minimality_rjt}, and is such that the implications in~\eqref{eq:rjtInclusions} are equivalence.

Furthermore, Steps~\ref{step:priority} and~\ref{step:priority2} enable to ensure that, when there is no path between a vertex $u \in \Va$ and a vertex $v\in \Vs$, then $u$ is placed before $v$ in the reverse topological ordering computed by the breadth first search. 
Therefore, $\preceq$ is a topological ordering on the graph $G''$ used as Step~\ref{step:topo} of Algorithm~\ref{alg:buildSolubleRjt}. 
Hence, if $G$ is soluble, Algorithm~\ref{alg:build_smallest_RJT2} builds a \rjt such that $\Gfree = G$.


Remark that on non-soluble IDs, Steps~\ref{step:priority} and~\ref{step:priority2} are a heuristic aimed at minimizing the size of $\B{C_v}$ for each $v$ in $\Vs$.
Such a heuristic is not relevant if valid cuts~\eqref{eq:validInequality} are not used.
In that case, an alternative strategy could be to add as few variable as possible to $C_v$ for $v$ in $\Va$ to improve the quality of the soluble relaxation $\overline{G}$.
This could be done by putting vertices $u$ in $\Vs$ unrelated to $v \in \Va$ after in this topological order, i.e., by replacing $\Va$ by $\Vs$ in Steps~\ref{step:priority} and~\ref{step:priority2}.

\section{McCormick Relaxation}
\label{app:McCormick}

McCormick inequalities allow to turn the NLP formulation~\eqref{pb:NLP}
into the MILP formulation~\eqref{pb:MILP}. Further good bounds ease the resolution
of Problem~\eqref{pb:MILP}. In this section we first discuss these relaxation,
show that in the NLP formulation loose bounds are useless while tight bounds improve
the formulation. Finally, we give an algorithm to compute good quality bounds.

\subsection{Review of McCormick's relaxation}
For the sake of completeness we briefly recall McCormick's relaxation,
and condition for exactness if all of the variables but one are binary.

\begin{prop}
Consider the variables $(x,y,z) \in [0,1]^3$ and the following constraint
\begin{equation}
\label{eq:prod_cst}
	z = x y
\end{equation}
Further, assume that we have an upper bound $y \leq b$.
We call ${\rm {McCormick}}\eqref{eq:prod_cst}$ the following set of contraints
\begin{subequations}
\label{eq:generic_McCormick}
\begin{align}
z \geq\;& y + xb - b \label{cst:mccormick_b}\\
z \leq\;& y  \label{cst:mccormick_a}\\
z \leq\;& bx \label{cst:mccormick_c}
\end{align}
\end{subequations}

If $x,y$ and $z$ satisfy \cref{eq:prod_cst}, then they also satisfy \cref{eq:generic_McCormick}.
If $x$ is a binary variable (that is $x \in \{0,1\}$) and \cref{eq:generic_McCormick} is satisfied,
then so is \cref{eq:prod_cst}.
\end{prop}
\begin{proof}
	Consider $x \in [0,1]$, $y \in [0,b]$ and $z \in [0,1]$, such that $z = x y$.
	Noting that $(1-x)(b-y) \geq 0$ we obtain Constraint~\eqref{cst:mccormick_b}.
	Constraints~\eqref{cst:mccormick_a} and~\eqref{cst:mccormick_c} are obtained
	by upper bounding by bounding one variable.
Now assume that $x \in \{0,1\}$, $y \in [0,b]$ and $z\in[0,1]$ satisfy \cref{eq:generic_McCormick}.
	Then, if $x =1$, constraints~\eqref{cst:mccormick_b} and~\eqref{cst:mccormick_a}
	yield $z=y$. Otherwise, as $z \geq 0$, we have $z=0$ by~\eqref{cst:mccormick_c}.
\end{proof}


\subsection{Choice of the bounds in McCormick inequalities}
\label{sec:bounds}


\subsubsection{Using $b_{\D{v}} = 1$ leads to loose constraints}
\label{sub:using_}

As $\mu_{\D{v}}$ is a probability distribution, $1$  is an immediate upper bound on $\mu_{\D{v}}$.
Let $\calQ^1$ be the polytope $\calQ^b$ obtained using bounds vector $b$ defined by $b_{\D{v}}=1$ for all $v$ in $\Va$.

\begin{prop}
Let $\mu$ be in $\overline{\calP}$. Then there exists $\delta$ in $\Delta$ such that $(\mu,\delta)$ belongs to $\calQ^1$, and the linear relaxation of~\eqref{pb:MILP} is equal to $\displaystyle\max_{\mu \in \overline{\calP}}\sum_{v \in \Vl} \langle r_v , \mu_v \rangle$.
\end{prop}
\begin{proof}
Let $v$ be a vertex in $\Va$, and let
$$\delta_{v|\prt{v}}(x_{\fa{v}}) = \left\{
\begin{array}{ll}
\frac{\mu_{\fa{v}}(x_{\fa{v}})}{\mu_{\prt{v}}(x_{\prt{v}})} &  \text{ if } \mu_{\prt{v}}(x_{\prt{v}}) \neq 0,
 \\
 \ind_{e_v}(x_v) & \text{otherwise,}
\end{array}
\right.  $$
where $e_v$ is an arbitrary element of $\calX_v$.
To prove the result, we show that \eqref{eq:McCormick} is satisfied for this well-chosen $\delta_{v|\prt{v}}$ and $b_{\D{v}}=1$.

We have
\begin{align*}
\mu_{C_v}(x_{C_v}) - \mu_{\D{v}}(x_{\D{v}})
&\geq \sum_{x'_{C_v \backslash \fa{v}} }\underbrace{\mu_{C_v}(x'_{C_v \backslash \fa{v}},x_{\fa{v}}) - \mu_{\D{v}}(x'_{C_v \backslash \fa{v}},x_{\prt{v}})}_{\leq 0} \\
&= \mu_{\fa{v}}(x_{\fa{v}}) - \mu_{\prt{v}}(x_{\prt{v}}) \\
&= \frac{1}{\mu_{\prt{v}}(x_{\prt{v}})}(\delta_{v|\prt{v}}(x_{\fa{v}}) - 1) \\
&\geq \delta_{v|\prt{v}}(x_{\fa{v}}) - 1
\end{align*}
which yields $\mu_{C_v} \geq \mu_{\D{v}}+ (\delta_{v|\prt{v}} -1) b_{\D{v}}$.

Besides, if $\mu_{C_v}(x_{C_v}) \geq 0$, following the definition of $\delta$ and given that $\mu_{\prt{v}}(x_{\prt{v}})\leq 1$, we have
$$ \delta_{v|\prt{v}}(x_{\fa{v}}) 
\geq \mu_{\fa{v}}(x_{\fa{v}}) \geq \mu_{C_v}(x_{C_v}),$$
and the constraint $\mu_{C_v} \leq \delta_{v|\prt{v}}  b_{\D{v}}$ is satisfied.

Finally, $\mu_{C_v} \leq \mu_{\D{v}}$ follows from the marginalization constraint $\mu_{\D{v}} = \sum_{x_v} \mu_{C_v}$  in the definition of the local polytope.
\end{proof}

\subsubsection{McCormick inequalities with well-chosen bounds are useful}
\label{sub:mccormick_inequalities_with_well_chosen_bounds_are_useful}
This section provides examples of IDs where McCormick inequalities~\eqref{eq:McCormick} improves the linear relaxation of MILPs~\eqref{pb:MILP} and \eqref{eq:MILPstrengthened}.

Consider the ID on Figure~\ref{fig:usefulMcCormick}.a,
and assume that we have a bound $\mu_{st} \leq b_{st}$.
Then, the McCormick relaxation of $\mu_{sta} = \mu_{st}\delta_{a|t}$
reads
\begin{equation}
\label{eq:McCormick_bound}
\begin{cases}
\mu_{sta} & \geq \mu_{st} + b_{st}(\delta_{a|t} - 1) \\
\mu_{sta} & \leq b_{st}\delta_{a|t}
\end{cases}
\end{equation}
Suppose that all variables are binary, that $s$ is Bernoulli with parameter $\frac{1}{2}$, that $\bbP(X_t = 1 | X_s) = 1+\varepsilon X_s - \varepsilon(1-X_s)$, that $X_w$ indicates if $X_s = X_a$, and that the objective is to maximize $\bbE_\delta(X_w)$, and has value $\tfrac{1}{2} + \varepsilon$.
An optimal policy consists in choosing $X_a = X_t$.
An optimal solution of the linear relaxation of \eqref{pb:MILP} on $\overline{\calP}$ without McCormick inequalities, has value 1.
Whereas an optimal solution with McCormick inequality and $b_{st}(x_s,x_t) = \tfrac{1}{2} + \varepsilon \ind_{x_s = x_t}$ has value $\tfrac{1}{2} + \varepsilon$.
However, on this simple example, the McCormick inequalities are implied by the valid inequalities of Section~\ref{sec:validCuts}.
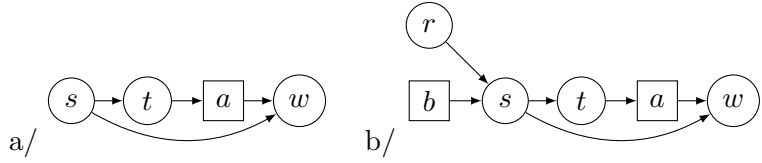
\begin{figure}
\begin{center}
a/
\begin{tikzpicture}
\node[sta] (s) at (0,0) {$s$};
\node[sta] (t) at (1,0) {$t$};
\node[act] (a) at (2,0) {$a$};
\node[sta] (w) at (3,0) {$w$};
\draw[arc] (s) -- (t);
\draw[arc] (t) -- (a);
\draw[arc] (a) -- (w);
\draw[arc] (s) to[bend right] (w);
\end{tikzpicture}
\quad b/
\begin{tikzpicture}
\node[act] (b) at (-1,0) {$b$};
\node[sta] (r) at (-1,1) {$r$};
\draw[arc] (b) -- (s);
\draw[arc] (r) to (s);
\node[sta] (s) at (0,0) {$s$};
\node[sta] (t) at (1,0) {$t$};
\node[act] (a) at (2,0) {$a$};
\node[sta] (w) at (3,0) {$w$};
\draw[arc] (s) -- (t);
\draw[arc] (t) -- (a);
\draw[arc] (a) -- (w);
\draw[arc] (s) to[bend right] (w);
\end{tikzpicture}
\end{center}
\caption{ID with useful McCormick inequalities.}
\label{fig:usefulMcCormick}
\end{figure}

This is no more the case on the ID of Figure~\ref{fig:usefulMcCormick}.b, where $r$ is a Bernoulli of parameter $0.5$ and $X_s = X_rX_b + (1-X_r)(1-X_b)$, and the remaining of the parameters are defined as previously. Using the same bounds, this new example leads to exactly the same results as before.


\subsection{Algorithm to compute good quality bounds}
\label{sub:algorithm_to_choose_good_quality_bounds}


This section provides a dynamic programming equation to compute bounds $b_{\D{v}}$ on $\mu_{\D{v}}$ that are smaller than $1$.
Let $\calG$ be a \rjt, and $C_1, \ldots, C_n$ be a topological order on $\calG$. Let $C_{k}$ be a vertex in $\calG$, $C_j$ be the parent of $C_k$ and $C_i$ the parent of $C_j$ ($i < j < k$). If $k=1$, then $C_i=C_j=C_k=C_1$. We introduce the notation $C_j^a = (C_j \backslash (C_i \cup C_k))\cap \Va$.
We define inductively on $k$ the functions $\tilde{b}_k : \calX_{C_k} \rightarrow [0,1]$ as follows. 
\begin{equation}
	\left\{
	\begin{array}{ll}
	\tilde{b}_{1}(x_{C_1}) &= \displaystyle \prod_{v \in C_1 \cap \Vs} p(x_v|x_{\prt{v}}) \\
	\tilde{b}_k(x_{C_k}) &= \displaystyle \bigg( \sum_{x_{\prt{C_j^a}}}\max_{x_{C_j^a}} \sum_{x_{(C_i \cup C_j) \backslash (C_k \cup C_j^a)}} \tilde{b}_{i}(x_{C_{i}}) \prod_{v\in ((C_j \cup C_k) \backslash C_i) \cap \Vs} p(x_v | x_{\prt{v}}) \bigg) \label{eq:bound_res_i}\\
	\text{for }k>1
	\end{array}
	\right.
\end{equation}

\begin{prop}
Let $\mu$ be in $\calS(G)$.
We have $\mu_{C_k}(x_{C_k}) \leq \tilde{b}_k(x_{C_k})$ for all $i$ and $x_{C_k}$ in $\calX_{C_k}$.
\end{prop}
As a consequence, $b_{\D{v}}$ defined as $\sum_{x_v}\tilde{b}_{C_v}$ provides an upper bound on $\mu_{\D{v}}$ that can be used in McCormick constraints.

\begin{proof}
We prove the result by induction. Let $(\mu, \delta)$ be a feasible solution of Problem~\ref{pb:NLP} and $C_j$ be the parent of $C_k$ in $\calG$, and $C_i$ the parent of $C_j$ ($i < j < k$). 

\noindent If $k=1$, then the result is obtained by using $\delta_v \leq 1$ for all $v \in \Va$. 

\noindent We assume now that the induction is true until $k > 1$. We have
\begin{align}
\mu_{C_k}(x_{C_k}) &= \sum_{x_{(C_i \cup C_j) \backslash C_k}} \mu_{C_i}(x_{C_i}) \prod_{v \in ((C_j \cup C_k) \backslash C_i) \cap \Vs} p(x_v | x_{\prt{v}}) \prod_{v \in ((C_j \cup C_k) \backslash C_i) \cap \Va} \delta_v(x_v | x_{\prt{v}}) \label{eq:bounds_ex1}\\
				   &\leq \max_{\delta_{(C_j \cup C_k) \backslash C_i}} \sum_{x_{(C_i \cup C_j) \backslash C_k}} \mu_{C_i}(x_{C_i}) \prod_{v \in ((C_j \cup C_k) \backslash C_i) \cap \Vs} p(x_v | x_{\prt{v}}) \prod_{v \in ((C_j \cup C_k) \backslash C_i) \cap \Va} \delta_v(x_v | x_{\prt{v}}) \label{eq:bounds_ex2}\\
				   &\leq \max_{\delta_{C_j \backslash (C_i \cup C_k)}} \sum_{x_{(C_i \cup C_j) \backslash C_k}} \mu_{C_i}(x_{C_i}) \prod_{v \in ((C_j \cup C_k) \backslash C_i) \cap \Vs} p(x_v | x_{\prt{v}}) \prod_{v \in (C_j \backslash (C_i\cup C_k)) \cap \Va} \delta_v(x_v | x_{\prt{v}}) \label{eq:bounds_ex3} \\
				   &\leq \max_{\delta_{C_j \backslash (C_i \cup C_k)}} \sum_{x_{(C_i \cup C_j) \backslash C_k}} b_{C_i}(x_{C_i}) \prod_{v \in ((C_j \cup C_k) \backslash C_i) \cap \Vs} p(x_v | x_{\prt{v}}) \prod_{v \in (C_j \backslash (C_i\cup C_k)) \cap \Va} \delta_v(x_v | x_{\prt{v}}) \label{eq:bounds_ex4}
\end{align}
From \eqref{eq:bounds_ex1} to \eqref{eq:bounds_ex2}, we maximize over the policies in $(C_i \cup C_j) \backslash C_k$. From \eqref{eq:bounds_ex2} to \eqref{eq:bounds_ex3}, we bound all policies in $C_k \cap \Va$ by $1$. Then \eqref{eq:bounds_ex4} is obtained by using the induction assumption. Let $\alpha : \calX_{\fa{C_j^a}} \rightarrow \bbR$ be such that for all $x_{C_j^a} \in \calX_{C_j^a}$,
$$\alpha(x_{C_j^a}) = \sum_{x_{(C_i \cup C_j) \backslash (C_k \cup \fa{C_j^a})}} \tilde{b}_{C_i}(x_{C_i}) \prod_{v \in ((C_j \cup C_k) \backslash C_i) \cap \Vs} p(x_v | x_{\prt{v}}).$$
Then, \eqref{eq:bounds_ex4} becomes
\begin{align}
	\mu_{C_k}(x_{C_k}) &\leq \max_{\delta_{C_j^a}} \sum_{x_{\fa{C_j^a}}} \alpha(x_{\fa{C_j^a}}) \prod_{v \in C_j^a} \delta_v(x_v | x_{\prt{v}}) \label{eq:bounds_ex5}
\end{align}
\noindent Now, we can suppose that $\offspring{C_v} = \{ v\}$. Therefore, $\vert C_j^a \vert \leq 1$ and the maximum above can be decomposed into the sum. 
\begin{align}
	\mu_{C_k}(x_{C_k}) &\leq \sum_{x_{\prt{C_j^a}}}\max_{\delta_{C_j^a}} \sum_{x_{C_j^a}} \alpha(x_{\fa{C_j^a}}) \prod_{v \in C_j^a} \delta_v(x_v | x_{\prt{v}}) \label{eq:bounds_ex6} \\
					   &\leq \sum_{x_{\prt{C_j^a}}}\max_{x_{C_j^a}} \alpha(x_{\fa{C_j^a}}) \label{eq:bounds_ex7}
\end{align}
where from \eqref{eq:bounds_ex6} to \eqref{eq:bounds_ex7} we use a local maximization. Therefore, we obtain the result
\begin{align*}
	\mu_{C_k}(x_{C_k}) &\leq \sum_{x_{\prt{C_j^a}}}\max_{x_{C_j^a}} \sum_{x_{(C_i \cup C_j) \backslash (C_k \cup \fa{C_j^a})}} \tilde{b}_{C_i}(x_{C_i}) \prod_{v \in ((C_j \cup C_k) \backslash C_i) \cap \Vs} p(x_v | x_{\prt{v}}) \label{eq:bounds_ex8}
\end{align*}
\end{proof}
Note that $\tilde{b}_k(x_{C_k})$ is computed via an order two recursion from $\tilde{b}_i(x_{C_i})$ where $i$ is the grand-parent of $k$, which can be generalized to higher order if stricter bound are needed.

\end{document}